\documentclass[11pt]{amsart}

 \usepackage{amsfonts,graphics,amsmath,amsthm,amsfonts,amscd,amssymb,amsmath,latexsym,multicol,
 mathrsfs}
\usepackage{epsfig,url}
\usepackage{flafter}
\usepackage{fancyhdr}
\usepackage{hyperref}
\hypersetup{colorlinks=true, linkcolor=black}

 \usepackage[matrix, arrow]{xy}

\DeclareMathOperator{\Div}{Div}

\DeclareMathOperator{\Proj}{Proj}

\DeclareMathOperator{\Spec}{Spec}
\DeclareMathOperator{\Supp}{Supp}

\DeclareMathOperator{\vol}{vol}

 \numberwithin{equation}{subsection}
 \numberwithin{footnote}{subsection}

 \newtheorem{lem}[subsection]{Lemma}
 \newtheorem{prop}[subsection]{Proposition}
 \newtheorem{thm}[subsection]{Theorem}

{
\theoremstyle{upright}

}

 \newcommand{\N}{\mathbb N}
 \newcommand{\PP}{\mathbb P}
 \newcommand{\A}{\mathbb A}
 \newcommand{\Q}{\mathbb Q}
 \newcommand{\R}{\mathbb R}

 \newcommand{\bir}{\dashrightarrow}

\title{\large T\MakeLowercase{oroidal and toric models of fibrations over curves}}
\author{\large C\MakeLowercase{aucher} B\MakeLowercase{irkar}}
\date{\today}
\begin{document}
\maketitle

\begin{abstract}
We construct relatively bounded toroidal and toric models of relatively bounded fibrations over curves.  
\end{abstract}

\tableofcontents


\section{\bf Introduction}

We work over an algebraically closed field $k$ of characteristic zero unless stated otherwise.
The aim of this paper is to construct relatively bounded toroidal and toric models of relatively bounded fibrations over curves. These constructions are crucially needed in [\ref{B-sing-Fano-fib-beyond}] to prove several conjectures in birational geometry, for example, conjectures of Shokurov and M$^{\rm c}$Kernan on singularities on Fano type fibrations. In recent years, toroidal and toric methods have been applied in other places in the study of Fano varieties and singularities, e.g. [\ref{BQ}][\ref{B-BAB}]. One of the key ingredients for allowing the use of toroidal methods is boundedness of complements [\ref{B-Fano}].

Consider a family of fibrations $f\colon X\to Z$ over curves, i.e. $f$ is a contraction from a variety onto a smooth curve. Assume that the family is relatively bounded (see \ref{ss-couples} for definitions). Our aim is to change these fibrations into new fibrations which are toroidal and still relatively bounded. 
The techniques developed in [\ref{KKMB}] are enough to produce toroidal fibrations but without the relative boundedness. 
Indeed, to apply this approach, one takes a resolution $W\to X$ so that all the fibres of $W\to Z$ have simple normal crossings singularities, and then constructs an appropriate cover of $W$. Taking the resolution we lose control of the relative boundedness. In fact, even for $\dim X=2$, it is not difficult to construct examples such that any choice of resolution would not satisfy relative boundedness. So we cannot use this approach. Instead, in this paper, we use the technique of families of nodal curves developed by de Jong [\ref{de-jong-smoothness-semi-stability}][\ref{de-jong-nodal-family}]. 
\bigskip

{\textbf{\sffamily{Toroidal models.}}}
Here is a precise formulation of existence of relatively bounded toroidal models for relatively
bounded fibrations over curves. This allows one to reduce problems in general settings to problems in toroidal settings (see [\ref{B-sing-Fano-fib-beyond}, \S 6] for more on this). For related definitions, see Sections 2 and 3; for the definition
of good tower of families of split nodal curves, see \ref{ss-tower-family-nodal-curves}. 

\begin{thm}\label{t-bnd-torification}
Let $d,r\in \N$. Then there exists an $r'\in \N$ depending only on $d,r$ satisfying the 
following. Assume that 
\begin{itemize}
\item $(X,D)$ is a couple of dimension $d$, 

\item $f\colon X\to Z$ is a projective morphism onto a smooth curve, 

\item $z\in Z$ is a closed point, and 

\item $A$ is a very ample$/Z$ divisor on $X$ such that $\deg_{A/Z}A\le r$ and $\deg_{A/Z}D\le r$.
\end{itemize}
Then, perhaps after shrinking $Z$ around $z$, there exists a commutative diagram 
$$
 \xymatrix{
 (X',D') \ar[d]^{f'} \ar[r]^{\pi} &    (X,D) \ar[d]^f \\
  (Z',E') \ar[r]^{\mu} & Z
  } 
$$  
of couples and a very ample$/Z'$ divisor $A'$ on $X'$ such that 
\begin{itemize}
\item  $(X',D')\to (Z',E')$ is a toroidal morphism, and in case $d\ge 2$, it factors as a good tower 
$$
(X_d',D_d') \to \cdots \to (X_1',D_1')
$$
of families of split nodal curves, where $(X',D')=(X_d',D_d')$ and $(Z',E')=(X_1',D_1')$,

\item  $\pi$ and $\mu$ are alterations,

\item 
$
\deg_{A'/Z'}A'\le r', \ \  \deg_{A'/Z'}D'\le r', \ \ \deg \pi\le r', \ \ \mbox{and} \ \ \deg \mu\le r',
$

\item the induced morphism 
$$
\pi|_{X'\setminus D'}\colon {X'\setminus D'}\to {X\setminus D}
$$ 
is quasi-finite,

\item $D'$ contains the fibre of $X'\to Z$ over $z$,

\item there is a Cartier divisor $G'\ge 0$ on $X'$ such that $A'-G'$ is ample$/Z'$ and $\Supp G'=D'$, and 

\item $A'-\pi^*A$ is ample over $Z'$.
\end{itemize}
\end{thm}

\bigskip

{\textbf{\sffamily{Toric models.}}}
Our next result aims to construct toric models of fibrations over curves, again keeping relative boundedness. 
The importance of this is that it allows one to reduce problems in toroidal settings to problems in toric settings (see [\ref{B-sing-Fano-fib-beyond}, \S7] for more on this).

\begin{thm}\label{p-local-desc-torif-bnd-fib-II}
Let $d,r$ be natural numbers. 
Assume  $(X,D)$ and $X\to Z$ satisfy the assumptions of Theorem \ref{t-bnd-torification} with the given $d,r$. 
Then we can choose 
$(X',D')\to (Z',E')$ and $r'\in \N$ in the theorem so that if $x'\in X'$ is a closed point and $z'\in Z'$ is its image, then  
perhaps after shrinking $Z'$ around $z'$, we can find a commutative diagram of varieties and couples
$$
\xymatrix{
& {M}'\ar[ld] \ar[rd] & &N' \ar[ll]\ar[d]\\
(X',D') \ar[rd] & & ({Y}',{L}')\ar[ld] \ar@{-->}[r] & P'=\PP^{d-1}_{Z'}\ar[lld]\\
&(Z',E') &&
}
$$
where
\begin{enumerate}
\item all arrows are projective morphisms, except that $Y'\bir P'$ is a birational map,

\item $N'\to {M}'$ is birational and $N'\to P'$ is an alteration, 

\item $M'\to X'$ and $M'\to Y'$ are \'etale at some closed point $m'$ mapping to $x'$, 

\item the inverse images of $D'$ and $L'$ to $M'$ coincide near $m'$,

\item  if $G'$ is the sum of the coordinate hyperplanes of $\PP^{d-1}_{Z'}$ and the inverse image of $E'$, then the induced map $P'\setminus G'\bir Y'$ is an open immersion,

\item $(Y',L')$ is lc near $y'$, the image of $m'$, and any lc place of $(Y',L')$ with centre at $y'$ is an lc place of $(P',G')$, and

\item there is an ample$/Z'$ Cartier divisor $H'$ on ${Y}'$ such that 
$$
\vol_{/Z'}(A'|_{N'}+H'|_{N'}+G'|_{N'})\le r'
$$ 
where $A',r'$ are as in Theorem \ref{t-bnd-torification}.
\end{enumerate}
\end{thm}

Using the diagram in Theorem \ref{p-local-desc-torif-bnd-fib-II}, problems on $X'$ near $x'$ can be translated into problems on $Y'$ near $y'$ via $M'$. 
Here $Y'$ is toric near $y'$ over some formal neighbourhood of $z'$, the image of $y'$ in $Z'$. In turn, the birational map $Y'\bir P'$ is used to further translate those problems into problems about $P'$ which is toric (not just locally) over a formal neighbourhood of $z'$. Since $Z'$ is a smooth curve, one can pretend that it is just $\A^1$, hence translate the problems into genuinely toric problems. Indeed, this is how Theorem \ref{p-local-desc-torif-bnd-fib-II} is used in [\ref{B-sing-Fano-fib-beyond}]. 

General (weak) semi-stable
reductions have been developed in recent years, e.g. [\ref{Qu}] (which is applied in [\ref{BQ}]), relying
on log geometry. It is likely that this can be used to get alternative proofs of Theorems \ref{t-bnd-torification} and
\ref{p-local-desc-torif-bnd-fib-II} but it would not be straightforward and requires work.
\bigskip

{\textbf{\sffamily{Acknowledgements.}}}
This work was partially done at the University of Cambridge. It was completed at Tsinghua University with support of a grant from Tsinghua University and a grant of the National Program of Overseas High Level Talent. Thanks to Santai Qu, Roberto Svaldi, and the referees for their valuable comments.
And thanks to the participants of activities devoted to this work, including a workshop in June 2023 and a seminar series in March--May 2024 at Tsinghua University and a workshop in May 2024 at Fudan University. 


\section{\bf Preliminaries}

We will work over an algebraically closed field $k$ of characteristic zero unless stated otherwise. Varieties are reduced, irreducible, and quasi-projective.

\subsection{Morphisms}

An \emph{alteration} is a surjective projective morphism $Y\to X$ of varieties of the same dimension, hence it is generically finite. 
A \emph{contraction} is a projective morphism $f\colon X\to Z$ of schemes with $f_*\mathcal{O}_X=\mathcal{O}_Z$, hence it is surjective with connected fibres.  

Given a morphism $g\colon Y\to X$ of schemes and a subset $T\subset X$, $g^{-1}T$ denotes the set-theoretic inverse image of $T$. If $T$ is a closed subscheme, we then consider $g^{-1}T$ with its induced reduced scheme structure. But if we consider the scheme-theoretic inverse image of $T$, we will say so explicitly. 

\subsection{Divisors, degree, and volume}\label{ss-divs-deg}
Let $X$ be a normal variety and let $D$ be an $\R$-divisor. For a prime divisor $T$ on $X$, $\mu_TD$ denotes the coefficient of $T$ in $D$. If $D$ is $\R$-Cartier and if $T$ is a prime divisor over $X$, i.e., on some birational modification $g\colon Y\to X$, then by $\mu_TD$ we mean $\mu_Tg^*D$. Here and elsewhere,
by a \emph{birational modification}, we mean a birational contraction $Y\to X$ from a normal variety.

Let $f\colon X\to Z$ be a surjective projective morphism of varieties. 
Let $A$ be a $\Q$-Cartier divisor on $X$. 
For a Weil divisor $D$ on $X$ we define the \emph{relative degree} of $D$ over $Z$ with respect to $A$ 
as 
$$
\deg_{A/Z}D:=(D|_F)\cdot (A|_F)^{n-1}
$$ 
where $F$ is a general fibre of $f$ and $n=\dim F$. 
It is clear that this is a generic degree,  so the vertical$/Z$ components of $D$ do not contribute 
to the degree. Note that $F$ may not be irreducible: by a general fibre we mean fibre over a general point of $Z$.
In practice, we take $A$ to be ample over $Z$.
A related notion is the \emph{relative volume} of $D$ over $Z$ which we define as $\vol_{/Z}(D):=\vol(D|_F)$.

For a morphism $g\colon V\to X$ of varieties (or schemes) and an $\R$-Cartier $\R$-divisor $N$ on $X$, 
we sometimes write $N|_V$ instead of $g^*N$. 

\subsection{Pairs and singularities}\label{ss-pairs}
A \emph{pair} $(X,B)$ consists of a normal variety $X$ and an $\R$-divisor $B\ge 0$ such that $K_X+B$ is $\R$-Cartier. We call $B$ the \emph{boundary divisor}.

Let $\phi\colon W\to X$ be a log resolution of a pair $(X,B)$. Let $K_W+B_W$ be the 
pullback of $K_X+B$. The \emph{log discrepancy} of a prime divisor $D$ on $W$ with respect to $(X,B)$ 
is defined as 
$$
a(D,X,B):=1-\mu_DB_W.
$$
A \emph{non-klt place} of $(X,B)$ is a prime divisor $D$ over $X$, that is, 
on some birational modification of $X$,  such that $a(D,X,B)\le 0$, and a \emph{non-klt centre} is the image of such a $D$ on $X$. 

We say $(X,B)$ is \emph{lc} (resp. \emph{klt})(resp. \emph{$\epsilon$-lc}) if 
 $a(D,X,B)\ge 0$ (resp. $>0$)(resp. $\ge \epsilon$) for every $D$. This means that  
every coefficient of $B_W$ is $\le 1$ (resp. $<1$)(resp. $\le 1-\epsilon$). 
Note that since $a(D,X,B)=1$ for most prime divisors, we necessarily have $\epsilon\le 1$.

A \emph{log smooth} pair is a pair $(X,B)$ where $X$ is smooth and $\Supp B$ has simple 
normal crossing singularities. Assume $(X,B)$ is a log smooth pair and assume $B=\sum_{i=1}^r B_i$ is reduced  
where $B_i$ are the irreducible components of $B$. 
A \emph{stratum} of $(X,B)$ is an irreducible component of $\bigcap_{i\in I}B_i$ for some non-empty $I\subseteq \{1,\dots,r\}$.  
Since $B$ is reduced, a stratum is nothing but a non-klt centre of $(X,B)$.

\subsection{Fibre products}

\begin{lem}\label{l-fibred-product}
Let $Z,V$ be schemes over a scheme $T$. Assume $W$ is a closed subscheme of $V$ and that the induced morphism $Z\times_TV\to V$ factors through the closed immersion $W\to V$. Then the induced morphism $Z\times_TW\to Z\times_T V$ is an isomorphism.  
\end{lem}
\begin{proof}
Considering the morphisms $Z\times_TW \to Z$ and $Z\times_TW\to W \to V$, and the universal property of $Z\times_TV$, we get an induced morphism 
$$
f\colon Z\times_TW\to Z\times_TV.
$$ 
On the other hand, considering $Z\times_TV\to Z$ and the assumed morphism
$Z\times_TV\to W$ factoring $Z\times_TV\to V$, and the universal property of $Z\times_TW$, we see that there is an induced morphism 
$$
g\colon Z\times_TV\to Z\times_TW.
$$ 
But then by the universal property of $Z\times_TV$, the composition $f\circ g$ is the identity morphism. Similarly, $g\circ f$ is also the identity morphism, hence both $f,g$ are isomorphisms.
\end{proof}

\subsection{Factoring morphisms}

\begin{lem}\label{l-produce-fib-rel-dim-one} 
Let $X\to Z$ be a surjective projective morphism between varieties, of relative dimension $\ge 1$, i.e. a general fibre of $X\to Z$ has dimension $\ge 1$. Then there exists a resolution $X'\to X$ so that the induced morphism $X'\to Z$ factors through a contraction $X'\to W/Z$ of relative dimension one.   
\end{lem}
\begin{proof}
Since $X\to Z$ is projective, it factors through a closed immersion $X\to \PP^n_Z$ followed by the projection $\PP^n_Z\to Z$. Changing the coordinates of $\PP^n$ and dropping one of them, we get a dominant rational map $\PP^n_Z\bir \PP^{n-1}_Z$ inducing a rational map $X\bir \PP^{n-1}_Z$. Let $Y$ be the image of the latter map. If $\dim X>\dim Y$, then take a resolution $X'\to X$ so that the induced map $X'\bir Y$ is a morphism, and let $X'\to W\to Y$ be the Stein factorisation of $X'\to Y$. Then $X'\to W$ is a contraction of relative dimension one, and $X'\to Z$ factors through $X'\to W$. Now assume that $X\bir Y$ is of relative dimension zero. Then we  consider a rational map $\PP^{n-1}_Z\bir \PP^{n-2}_Z$ and argue similarly and so on.    
\end{proof}

\subsection{{\'Etale morphisms}}

\begin{lem}\label{l-etale-morphism-regular-points}
Assume that $\pi\colon X\to Y$ is an \'etale morphism between normal varieties, and that $\gamma$ is a rational function on $Y$. Then $\pi^*(\gamma)$ is regular at a closed point $x\in X$ iff $\gamma$ is regular at $\pi(x)$.
\end{lem}
\begin{proof}
If $\gamma$ is regular at $y=\pi(x)$, then obviously $\pi^*(\gamma)$ is regular at $x$. Conversely, assume $\pi^*(\gamma)$ is regular at $x$. If $\gamma$ is not regular at $y$, then $\Div(\gamma)$ has a component with negative coefficient at $y$, hence since $\pi$ is \'etale, $\Div(\pi^*(\gamma))=\pi^*\Div(\gamma)$ has a component with negative coefficient at $x$. This is not possible, so $\gamma$ is regular at $y$.  
\end{proof}

\begin{lem}\label{l-special-subvar-etale}
Assume that $\pi\colon X\to Y$ is an \'etale morphism between normal varieties. Also assume that $\gamma$ is a regular function on $Y$ and $g$ is a nowhere vanishing regular function on $X$.  
Consider the closed subschemes 
$$
S\subset X\times \Spec k[\alpha_1,\alpha_2], \ \ T\subset Y\times\Spec k[\beta_1,\beta_2]
$$
defined by the equations $\alpha_1\alpha_2-\pi^*(\gamma)g=0$ and $\beta_1\beta_2-\gamma=0$, respectively.
Then the morphism  
$$
\phi\colon X\times \Spec k[\alpha_1,\alpha_2] \to Y\times\Spec k[\beta_1,\beta_2]
$$
which sends a closed point $(x,a,b)$ to the point $(\pi(x),a,\frac{b}{g(x)})$ induces a natural isomorphism $S\to X\times_YT$, hence inducing an \'etale morphism $S\to T$.
\end{lem}
\begin{proof}
The morphism $\phi$ decomposes as  
$$
X\times \A^2 \overset{\rho}\to X\times \A^2 \overset{\psi}\to Y\times \A^2
$$
where $\rho$ sends $(x,a,b)$ to $(x,a,\frac{b}{g(x)})$ and $\psi$ sends $(x,a,b)$ to $(\pi(x),a,b)$. Here $\rho$ is an isomorphism as $g$ is nowhere vanishing, and $\psi$ is induced by base change via $\pi$. The scheme-theoretic inverse image of $T$ under $\psi$ is just $X\times_YT$, and the scheme-theoretic inverse image of $T$ under $\phi$ is just $S$ because 
$$
\phi^*(\beta_1\beta_2-\gamma)=\alpha_1\frac{\alpha_2}{g}-\pi^*(\gamma)=\frac{1}{g}(\alpha_1\alpha_2-\pi^*(\gamma)g)
$$ 
and $g$ is nowhere vanishing. Therefore, we get $S\to X\times_YT\to T$ where the former morphism is 
an isomorphism and the latter is \'etale as $\pi$ is \'etale. 
\end{proof}

\begin{lem}\label{l-desced-prime-div}
Assume that $Y\to X$ is a dominant morphism of varieties, which is \'etale at a closed point $y\in Y$. Assume that $D$ is a prime divisor over $Y$ with centre passing through $y$. Then we can find resolutions $Y'\to Y$ and $X'\to X$ so that the induced map $Y'\bir X'$ is a morphism, $D$ is a divisor on $Y'$, and the image of $D$ on $X'$ is a divisor.
\end{lem}
\begin{proof}
First, pick a resolution $X'\to X$, let $Y''$ be the main component of $Y\times_XX'$, and let $y''\in Y''$ be a closed point that maps to $y$ and is contained in the centre of $D$ on $Y''$. Then the induced map $Y''\to X'$ is \'etale at $y''$; in particular, $Y''$ is smooth at $y''$. Take a resolution $Y'\to Y''$ which is an isomorphism over $y''$. Replacing $Y\to X$ and $y$ with $Y'\to X'$ and $y''$, we can assume that $Y$ is smooth at $y$. If necessary, we replace $y$ by a general closed point of the centre of $D$ on $Y$.

Let $C$ be the centre of $D$ on $Y$ and let $E$ be the closure of the image of $C$ on $X$. Shrinking $Y,X$, we can assume that $Y,X,C,E$ are all smooth and that $Y\to X$ is \'etale. Let $X'\to X$ be the blowup of $X$ along $E$. Shrinking $Y$ and letting $Y'=Y\times_XX'$, the induced map $Y'\to Y$ is the blowup of $Y$ along $C$. Also $Y'\to X'$ is \'etale. Replace $Y\to X,y$ with $Y'\to X',y'$ where $y'\in Y'$ is a closed point mapping to $y$ and contained in the centre of $D$ on $Y'$. Repeat this process. By [\ref{kollar-mori}, Lemma 2.45], after finitely many steps, $D$ is a divisor on $Y$. Since $Y\to X$ is \'etale, the image of $D$ on $X$ is also a divisor.  
\end{proof}

\subsection{Coordinate hyperplanes}
By coordinate hyperplanes of 
$$
\PP^n=\Proj k[\beta_0,\dots,\beta_n]
$$ 
we mean the hyperplanes 
defined by vanishing of the $\beta_i$. When $Z$ is a variety, by coordinate hyperplanes on 
$\PP^n_Z=\PP^n\times Z$ we mean the pullback of the coordinate hyperplanes on 
$\PP^n$ via the projection $\PP^n_Z\to \PP^n$.


\section{\bf Couples and toroidal geometry}

\subsection{Couples}\label{ss-couples} 

A \emph{couple} $(X,D)$ consists of a variety $X$ and a reduced Weil divisor $D$ on $X$.
This is more general than the definition given in [\ref{B-sing-fano-fib}] because we are not assuming $X$ to be normal 
nor projective. Also note that a couple is not necessarily a pair in the sense that we are not assuming 
$K_X+D$ to be $\Q$-Cartier. In this paper, we often consider a couple $(X,D)$ equipped with a 
\emph{surjective} projective morphism $X\to Z$ in which case we denote the couple as $(X/Z,D)$ or $(X,D)\to Z$. 
We say a couple $(X/Z,D)$ is \emph{flat} if both $X\to Z$ and $D\to Z$ are flat. A couple $(X,D)$ is \emph{of dimension $d$} if $\dim X=d$.

Let $\mathcal{P}$ be a set of couples. We say $\mathcal{P}$ is \emph{generically relatively bounded} if 
there exist natural numbers $d,r$ such that for each $(X/Z,D)\in \mathcal{P}$ we have the following: $\dim X-\dim Z\le d$ and there is a very ample$/Z$ divisor $A$ on $X$ such that 
$$
\deg_{A/Z}A\le r ~~\mbox{and} ~~\deg_{A/Z}D\le r.
$$
If in addition all the $(X/Z,D)\in \mathcal{P}$ are flat, we say that $\mathcal{P}$ is \emph{relatively bounded}. 

When $D=0$ for every $(X/Z,D)\in \mathcal{P}$, we say $\mathcal{P}$ is a set of generically relatively  bounded (resp. relatively bounded) varieties.

\begin{lem}\label{l-bnd-couple-induced-by-fibration}
Let $W\to T$ be a projective morphism of varieties and $G$ an effective Cartier divisor on $W$. 
Let $\mathcal{P}$ be the set of couples $(Y/Z,E)$ satisfying the following: 
\begin{itemize}
\item $Z$ is a variety equipped with a morphism $Z\to T$,

\item $Y$ is an irreducible component of $Z\times_TW$ with reduced structure, mapping onto $Z$, 

\item the image of $Y\to W$ is not contained in $\Supp G$, and 

\item the horizontal$/Z$ part of $E$ is contained in $\Supp (G|_Y)$.
\end{itemize}
Then $\mathcal{P}$ is a generically relatively bounded set of couples.  
\end{lem}
\begin{proof}
Let $A$ be a very ample$/T$ divisor on $W$. Pick a sufficiently large $l$ so that $lA-G$ is very ample$/T$. 
Let $t\in T$ be a closed point and $W_t$ be the fibre of $W\to T$ over $t$. Assume $V\subset W_t$ is a union of  
irreducible components of $W_t$ of dimension $d$ with reduced structure, and that no component of $V$ is contained in $\Supp G$. 
Then $A|_{V}^{d}$ is bounded from above as the fibres $W_t$ belong to a bounded family, hence the left hand side of 
$$
A|_V^{d-1}\cdot G|_V\le A|_V^{d-1}\cdot lA|_V= l A|_{V}^{d}
$$  
is also bounded from above.

Let $z\in Z$ be a general closed point and $t\in T$ its image. Then each irreducible component of $Y_z$ is 
an irreducible component of the reduction of $(Z\times_TW)_z$: indeed, pick an open subset $U\subseteq Z\times_TW$ such that $U$ does not intersect any irreducible component of the reduction of $Z\times_TW$ other than $Y$; then since $z$ is general and hence $Y\to Z$ is flat over $z$, counting dimensions, we see that every irreducible component $R$ of $Y_z$ intersects $U_z$; and $U_z$ is an open subset of  $(Z\times_TW)_z$; hence $R$ is an irreducible component of $(Z\times_TW)_z$ with reduced structure. 

On the other hand,  
$(Z\times_TW)_z$ is isomorphic to $W_t$ which induces an isomorphism $Y_z\to V\subseteq W_t$ where $V$ is the union of some  
irreducible components of the reduction of $W_t$.  Since $Y$ is not mapped into $\Supp G$ and since $z$ is general, no component of $Y_z$ is contained in $\Supp (G|_Y)$, so no component of $V$ is contained in $\Supp G$. Moreover, $E_z$ is mapped to a reduced divisor $D$ on $V$ with $D\subseteq \Supp (G|_V)$. 

Now by the above arguments, 
$$
A|_{Y_z}^{d}=A|_{V}^{d} \ \ ~~~\mbox{and}~~~\ A|_{Y_z}^{d-1}\cdot E_z = A|_V^{d-1}\cdot D \le A|_V^{d-1}\cdot G|_V
$$ 
are bounded from above. So such $(Y/Z,E)$ form a generically relatively bounded set of couples.
\end{proof}

\subsection{Universal families of relatively bounded families of couples}

\begin{lem}\label{l-bnd-embed-proj-space}
Let $\mathcal{P}$ be a {relatively bounded} family of couples $(X/Z,D)$ where $Z$ is a smooth curve. 
Then there is a natural number $n$ 
depending only on $\mathcal{P}$ such that for each $(X/Z,D)\in \mathcal{P}$ and each closed point $z\in Z$, 
perhaps after shrinking $Z$ around $z$, the morphism $X\to Z$ factors as a closed immersion $X\to \PP^n_Z$ 
followed by the projection $\PP^n_Z\to Z$. 
\end{lem}
\begin{proof}
Pick $(X/Z,D)\in \mathcal{P}$. By assumption, there are fixed natural numbers $d,r$ such that $\dim X-\dim Z\le d$ and such that we can find a very ample$/Z$ divisor $A$ on $X$ with $\deg_{A/Z}A\le r$ and $\deg_{A/Z}D\le r$. 
Since $Z$ is a smooth curve, the sheaf $f_*\mathcal{O}_X(A)$ is locally free where $f$ denotes $X\to Z$. 
We can assume $Z=\Spec R$ and that $H^0(X,A)$ is a free $R$-module, say of rank $m+1$. 
Now since $f$ is projective and $A$ is very ample over $Z$, using a basis $\alpha_0,\dots,\alpha_m$ of $H^0(X,A)$, we can factor $f$ as a closed immersion $X\to \PP(H^0(X,A))$ followed by projection onto $Z$. Since $H^0(X,A)$ is a free $R$-module of rank $m+1$, $\PP(H^0(X,A))\simeq \PP^m_Z$. 

It is enough to show $m$ is bounded depending only on $\mathcal{P}$ as we can factor $\PP^m_Z\to Z$ as a closed immersion $\PP^m_Z\to \PP^n_Z$ followed by projection onto $Z$, for some fixed $n$. 
Let $F$ be a general fibre of $f$. By assumption, $A|_F^{e}=\deg_{A/Z}A\le r$  
where $e=\dim F\le d$. This implies that $F$ belongs to a bounded family (but $F$ may not be irreducible) and that 
$m=h^0(F,A|_F)-1$ is bounded from above. 
\end{proof}

\begin{lem}\label{l-univ-family}
Let $\mathcal{P}$ be a relatively bounded family of couples $(X/Z,D)$ where $Z$ is a smooth curve. Then there exist finitely many couples $({V}_i/T_i,C_i)$ satisfying the following. Assume $(X/Z,D)\in \mathcal{P}$. 
Then for each closed point $z\in Z$, perhaps after shrinking $Z$ around $z$, there exist an $i$ and a morphism $Z\to T_i$ such that 
$$
X=Z\times_{T_i}V_i \ \mbox{and} \ D=Z\times_{T_i}C_i.
$$
\end{lem}
\begin{proof}
 By Lemma \ref{l-bnd-embed-proj-space}, there is an $n$ depending only on $\mathcal{P}$ 
such that perhaps after shrinking $Z$ around $z$ the morphism $X\to Z$ factors through a 
closed immersion $X\to \PP^n_Z$. In particular, $X\to Z$ 
can be viewed as a flat family of closed subschemes $X_u$ of $\PP^n$, $u\in Z$, with finitely many possible 
Hilbert polynomials depending only on $\mathcal{P}$. Similarly, since $D\to Z$ is flat, it can be viewed as a flat family of  
closed subschemes $D_u$ of $\PP^n$ (of one dimension less), $u\in Z$, again with finitely many possible Hilbert polynomials depending only on $\mathcal{P}$. Below we will keep in mind that $D_u\subset X_u$. Shrinking $\mathcal{P}$, we can assume the Hilbert polynomials are fixed in each case.

By the existence of Hilbert schemes and their associated universal families, there are reduced schemes $R,S$ over $k$ 
and closed subschemes $W\subset \PP^n_R$ and $G\subset\PP^n_S$ such that the projections $W\to R$ and $G\to S$  
are flat, and if $(X/Z,D)\in \mathcal{P}$, 
then there are morphisms $Z\to R$ and $Z\to S$ inducing  
$$
X= Z\times_{R}W \ \mbox{and} \ D= Z\times_{S}G.
$$ 

Let $T=R\times S$, consider 
$$
V:=W\times S \ \mbox{and} \ C:=G\times R
$$ 
as closed subschemes of $\PP^n_T$, 
and consider the projections $V\to T$ and $C\to T$. If $(X/Z,D)$ is in $\mathcal{P}$, then the morphisms 
$Z\to R$ and $Z\to S$ determine a morphism $Z\to T$, and we can identify 
$$
X= Z\times_{T}V \ \mbox{and} \ D= Z\times_{T}C.
$$ 
Replace $T$ 
with the closure of the union of the images of the possible morphisms $Z\to T$ for all $(X/Z,D)\in\mathcal{P}$. 
Replace $V,C$ accordingly by base change (but at this point $V,C$ may not be reduced nor irreducible).

Pick $(X/Z,D)\in\mathcal{P}$ and let $Z\to T$ be the induced morphism. Then $Z$ is mapped into some irreducible 
component $T'$ of $T$. Let $V'\to T'$ and $C'\to T'$ be the induced families obtained by base change. Since $X$ is irreducible, 
$X\to V'$ maps $X$ into some irreducible component $V''$ of $V'$ and $X=Z\times_{T'}V''$,  by Lemma \ref{l-fibred-product}, where $V''$ is considered 
with its reduced structure. On the other hand, since $T'$ is irreducible and $C'\to T'$ is flat, every component of $C'$ 
is mapped onto $T'$. Let $C''$ be the reduction of $C'$. Since $D$ is reduced, $D=Z\times_{T'}C''$, by Lemma \ref{l-fibred-product}. 

The last paragraph shows that there are finitely many varieties $T_i$ and closed subschemes $V_i\subset \PP^n_{T_i}$ 
and $C_i\subset \PP^n_{T_i}$ where $V_i$ is integral, and $C_i$ is reduced and all of its irreducible components map onto $T_i$ such that for any $(X/Z,D)$ in $\mathcal{P}$, there is an $i$ such that 
$$
X=Z\times_{T_i}V_i \  \mbox{and} \ D=Z\times_{T_i}C_i.
$$
Also, for each $i$, there is a dense subset $L_i$ of closed points of $T_i$ such that 
for each $t\in L_i$ there is $(X/Z,D)$ in $\mathcal{P}$ 
so that $Z$ maps into $T_i$ and $t$ is in the image of $Z\to T_i$. 
In particular, if $V_{i,t}$ and $C_{i,t}$ are the fibres of $V_i\to T_i$ and $C_i\to T_i$ over $t$, 
then $C_{i,t}\subset V_{i,t}$ as $D\subset X$. Therefore, we can assume $C_i\subset V_i$ for every $i$ and view 
$C_i$ as a reduced divisor on $V_i$.
\end{proof}

\subsection{Morphisms and towers of couples}\label{ss-tower-couples}

(1) A \emph{morphism} $(Z,E)\to (V,C)$ between couples is a morphism $f\colon Z\to V$ such that 
$f^{-1}(C)\subseteq E$. 

(2) A \emph{tower of couples} is a sequence of morphisms of couples 
$$
(V_d,C_d)\to (V_{d-1},C_{d-1})\to \cdots \to (V_1,C_1)
$$
where each morphism $V_i\to V_{i-1}$ is \emph{dominant}. 

Given such a sequence, suppose in addition that $(Z_1,E_1)\to (V_1,C_1)$ 
is a morphism of couples such that over the generic point of $Z_1$ we have: $Z_1\times_{V_1}V_i$ is integral 
and not contained in $Z_1\times_{V_1}C_i$, for each $i$.  
We then define the \emph{pullback} of the tower  by base change to $(Z_1,E_1)$ as 
follows. Let $Z_i$ be the main component of $Z_1\times_{V_1}V_i$ and let $E_i$ be the 
codimension one part, with reduced structure, of the union of the inverse images of $C_i$ and $E_1$ under 
$Z_i\to V_i$ and $Z_i\to Z_1$, respectively. Note that if $C_i$ and $E_1$ are supports of 
effective Cartier divisors, then $E_i$ coincides with the union of the inverse images of $C_i$ and $E_1$, 
with reduced structure.

(3) We will not define birational maps between couples in general. But there will be a few instances in this paper in which we have two couples $(V,C)$ and $(V',C')$ together with a birational map $V\bir V'$ inducing an isomorphism $V\setminus C \to V'\setminus C'$. In this case, we say that we have a \emph{congruent birational map}  
$$
(V,C)\bir (V',C').
$$


\subsection{Toric geometry}\label{ss-toric-geometry}
We assume familiarity with toric geometry and 
we will follow [\ref{Cox-etal}] for basic concepts and results.
A \emph{toric variety} is a variety $X$ of dimension $d$ containing a torus $\mathbb{T}_X$ (that is, isomorphic to $(k^*)^d$) as an open subset so that the action of $\mathbb{T}_X$ on itself (induced by coordinate-wise multiplication of $(k^*)^d$) extends to an action on the whole $X$ [\ref{Cox-etal}, 3.1.1]. Here, $X$ is not necessarily normal. 
A \emph{toric morphism} $f\colon X\to Y$ between toric varieties is a morphism so that the restriction $f|_{\mathbb{T}_X}$ induces a morphism $\mathbb{T}_X\to \mathbb{T}_Y$ of algebraic groups and so that $f$ is equivariant with respect to the actions of the tori.  

A \emph{toric couple} is a couple $(X,D)$ where $X$ is a toric variety and $D$ is the union of the torus-invariant prime divisors on $X$. It is well-known that a normal toric couple $(X,D)$ corresponds to a certain combinatorial data (a fan structure) [\ref{Cox-etal}, Chapter 3] and that it is an lc pair [\ref{Cox-etal}, Chapter 11] satisfying $K_X+D\sim 0$.

\subsection{Formally Cartier divisors}\label{ss-formally-cartier}

Let $X$ be a variety, $x\in X$ be a closed point, and $\widehat{X}=\Spec \widehat{\mathcal{O}}_{X,x}$ where $\widehat{\mathcal{O}}_{X,x}$ denotes the completion of the local ring ${\mathcal{O}_{X,x}}$ with respect to its maximal ideal. The local ring ${\mathcal{O}}_{{X},x}$ is a G-ring (meaning Grothendieck ring) by [\ref{Matsumura}, Corollary and Remark 1 on page 259], so by definition of G-rings, the geometric fibres of $\widehat{X} \to \Spec {\mathcal{O}}_{{X},x}$ are regular: in the language of commutative algebra, this says that the homomorphism ${\mathcal{O}}_{{X},x}\to \widehat{\mathcal{O}}_{{X},x}$ is regular. 

Now assume $X$ is normal. Then $\widehat{\mathcal{O}}_{{X},x}$ is normal by the previous paragraph and [\ref{Matsumura}, Theorem 32.2] (or by [\ref{Zariski}]), hence $\widehat{X}$ is normal.
 Let $D$ be a Weil divisor on $X$. We define $\widehat{D}$ on $\widehat{X}$ as follows.  
Let $U$ be the smooth locus of 
$X$ and let $\widehat{U}$ be its inverse image in $\widehat{X}$, and $\pi\colon \widehat{U}\to U$ the induced 
morphism. Then $D|_U$ is Cartier and its pullback $\pi^*D|_U$ is a well-defined Cartier divisor. Now 
let $\widehat{D}$ be the closure of $\pi^*D|_U$ in $\widehat{X}$. Note that the complement of $\widehat{U}$ in $\widehat{X}$ has codimension at least two.

When $X$ is normal and $D$ is an effective Weil divisor on $X$, we can view $D$ as the closed subscheme of $X$ defined by the ideal sheaf $\mathcal{O}_X(-D)$ and think of $\widehat{D}$ as the corresponding closed subscheme of $\widehat{X}$, that is, if $D$ is given by an ideal $I$ near $x$, then $\widehat{D}$ is given by $\widehat{I}$.

\begin{lem}\label{l-formally-cartier}
Let $X$ be a normal variety, $x\in X$ be a closed point, and $\widehat{X}=\Spec \widehat{\mathcal{O}}_{X,x}$. 
Let $D$ be a Weil divisor on $X$ and let $\widehat{D}$ 
be the corresponding divisor on $\widehat{X}$. Then $D$ is Cartier near $x$ if and only if $\widehat{D}$ is Cartier. 
\end{lem}
\begin{proof}
If $D$ is Cartier near $x$, then  $\widehat{D}$ is Cartier. We show the converse. 
Shrinking $X$ and changing $D$ linearly, we can assume $D$ is effective, hence $\mathcal{O}_{X}(-D)\subset \mathcal{O}_X$.
Since $X$ is normal, $\mathcal{O}_{X}(-D)$ is a reflexive coherent sheaf. Since the morphism 
$\rho\colon \widehat{X}\to X$ is flat, $\rho^*\mathcal{O}_{X}(-D)$ is reflexive too [\ref{Hartshorne}, Proposition 1.8]. 
Moreover,  $\mathcal{O}_{\widehat{X}}(-\widehat{D})$ is reflexive, actually invertible, since $\widehat{D}$ is Cartier. 
Now as observed above, denoting the smooth locus of $X$ by $U$, $D|_U$ is Cartier and so is $\widehat{D}|_{\widehat{U}}$. Therefore, 
$\rho^*\mathcal{O}_{X}(-D)$ coincides with $\mathcal{O}_{\widehat{X}}(-\widehat{D})$ on $\widehat{U}$, hence $\rho^*\mathcal{O}_{X}(-D)$ and $\mathcal{O}_{\widehat{X}}(-\widehat{D})$ are equal as both are reflexive and as the complement of $\widehat{U}\subset \widehat{X}$ has codimension at least two [\ref{Hartshorne}, Proposition 1.6].  Thus $\rho^*\mathcal{O}_{X}(-D)$ is invertible, so applying [\ref{Matsumura}, Exercise 8.3] implies $\mathcal{O}_{X}(-D)$ 
is invertible near $x$, hence $D$ is Cartier near $x$.
\end{proof}

\subsection{Toroidal couples}

A couple $(X,D)$ is \emph{toroidal} at a closed point $x\in X$ 
if there exist a \emph{normal} toric variety $W$ and a closed point $w\in W$ such that there is 
a $k$-algebra isomorphism 
$$
\widehat{\mathcal{O}}_{{X},{x}}\to \widehat{\mathcal{O}}_{{W},w}
$$ 
of completion of local rings so that the ideal of $D$ is mapped to the ideal of the toric boundary divisor 
$C\subset W$ (that is, the complement of the torus). Then  
there is a common \'etale neighbourhood of $X,x$ and $W,w$ [\ref{Artin}, Corollary 2.6].
We call $(W,C),w$ a \emph{local toric model} of $(X,D),x$.
We say $(X,D)$ is toroidal if it is toroidal at every closed point.

Now let $f\colon (X,D)\to (Y,E)$ be a morphism of couples. Let $x\in X$ be a closed point 
and let $y=f(x)$. We say $(X,D)\to (Y,E)$ 
is a \emph{toroidal morphism} at $x$ if there exist local toric models $(W,C),w $ and $(V,B),v$ of 
$(X,D),x$ and $(Y,E),y$, respectively, and a toric morphism $W\to V$ of toric varieties 
inducing a commutative diagram 
$$
\xymatrix{
\widehat{\mathcal{O}}_{{X},{x}}\ar[r] & \widehat{\mathcal{O}}_{{W},w}\\
\widehat{\mathcal{O}}_{{Y},{y}} \ar[u] \ar[r] & \widehat{\mathcal{O}}_{{V},v} \ar[u]
}
$$
where the vertical maps are induced by the given morphisms and the horizontal maps are isomorphisms 
induced by the local toric models. We say the morphism of couples $f\colon (X,D)\to (Y,E)$ is toroidal 
if it is toroidal at every closed point. 

For a systematic treatment of toroidal couples, see [\ref{KKMB}]. 

\begin{lem}\label{l-toroidal-couple-lc}
Let $(X,D)$ be a toroidal couple. Then $X$ is normal and Cohen-Macaulay, $K_X+D$ is Cartier, and  
$(X,D)$ is an lc pair.
\end{lem}
\begin{proof}
Pick a closed point $x\in X$. Let $(W,C),w$ be a local toric model of $(X,D),x$. Since $W$ is toric and normal, it is Cohen-Macaulay.  
Thus $\widehat{\mathcal{O}}_{{W},w}$ is normal and Cohen-Macaulay, hence $\widehat{\mathcal{O}}_{{X},{x}}$ 
is normal and Cohen-Macaulay which implies $X$ is normal and Cohen-Macaulay at $x$, by [\ref{Bruns-Herzog}, Corollaries 2.1.8 and 2.2.23]. Alternative argument: ${\mathcal{O}}_{{X},x}$, ${\mathcal{O}}_{{W},w}$ are G-rings by [\ref{Matsumura}, Corollary on page 259], so by definition of G-rings, the homomorphisms ${\mathcal{O}}_{{X},x}\to \widehat{\mathcal{O}}_{{X},x}$ and ${\mathcal{O}}_{{W},w}\to \widehat{\mathcal{O}}_{{W},w}$ are regular; so by [\ref{Matsumura}, Theorem 32.2], 
${\mathcal{O}}_{{X},x}$ is normal (resp. regular, resp. Cohen-Macaulay, resp. reduced) iff $\widehat{\mathcal{O}}_{{X},x}$ is normal (resp. regular, resp. Cohen-Macaulay, resp. reduced) and a similar statement holds for ${\mathcal{O}}_{{W},w}$ and its completion.

Pulling back the canonical sheaf $\mathcal{O}_X(K_X)$ to $\Spec \widehat{\mathcal{O}}_{{X},{x}}$ gives the canonical sheaf of the latter [\ref{Bruns-Herzog}, Theorem 3.3.5]. In other words, $\widehat{K_X}$ is the canonical divisor of $\Spec \widehat{\mathcal{O}}_{{X},{x}}$ which is unique up to linear equivalence. Similarly, $\widehat{K_W}$ is the canonical divisor of $\Spec \widehat{\mathcal{O}}_{{W},{w}}$. Moreover, $(W,C)$ is toric, hence $K_W+C$ is Cartier near $w$. Thus using the given isomorphism $\widehat{\mathcal{O}}_{{X},{x}}\to \widehat{\mathcal{O}}_{{W},{w}}$ to identify the corresponding schemes, we deduce that $\widehat{K_X}+\widehat{D}\sim \widehat{K_W}+\widehat{C}$ is Cartier. Therefore, $K_X+D$ is Cartier near $x$, by Lemma \ref{l-formally-cartier}. 
Additionally, $(X,D)$ is lc at $x$ because $(W,C)$ is lc and because singularities are determined locally formally.
\end{proof}

We sketch an alternative approach to the second paragraph of the proof of the lemma. Applying [\ref{Artin}, Corollary 2.6], there is a common \'etale neighbourhood $U,u$ of $X,x$ and $W,w$. Assume that the inverse images of $D$ and $C$ to $U$ coincide near $u$ (this does not follow immediately from [\ref{Artin}, Corollary 2.6] but a modification of its proof should work; or one can apply [\ref{Denef}, Remark 2.4] and [\ref{KKMB}, page 195]; in this paper, however, we apply the lemma only in situations in which we will have a common \'etale neighbourhood satisfying the condition on inverse images). Then one can see quickly that, near $x$, $K_X+D$ is Cartier and  
$(X,D)$ is an lc pair.

\section{\bf Families of nodal curves and toroidalisation of fibrations}

The purpose of this section is to prove Theorem \ref{t-bnd-torification}.

\subsection{Families of nodal curves}

We now define families of nodal curves following [\ref{de-jong-smoothness-semi-stability}, 2.21-22].  
Note that these are called semi-stable curves in  [\ref{de-jong-smoothness-semi-stability}].

A \emph{nodal curve} over a field $K$ is a scheme $F$, projective over $K$, such that $F_{\overline{K}}$ is a 
connected reduced scheme of pure dimension one having at worst ordinary double point singularities where $F_{\overline{K}}$ 
means the scheme obtained after base change to the algebraic closure ${\overline{K}}$. We say 
$F$ is a \emph{split nodal curve} over $K$ if it is a nodal curve over $K$, that its irreducible components are geometrically 
irreducible and smooth over $K$, and that its singular points are $K$-rational (here singular points are points where 
$F\to \Spec K$ is not smooth). 

Now let $Y$ be a scheme. 
A \emph{family of (split) nodal curves} over $Y$ is a 
flat projective morphism $f\colon X\to Y$ of schemes such that for each $y\in Y$ the fibre $F$ over $y$ is a (split) nodal curve over the residue field $k(y)$.

\begin{lem}\label{l-nodal-curves-base-change}
Let $Y'\to Y$ be a morphism of schemes over the ground field $k$. Assume $f\colon X\to Y$ is a family of (split) nodal curves over $Y$, 
and let $X'=Y'\times_YX$. Then the induced morphism $f'\colon X'\to Y'$ is a family of (split) nodal curves over $Y'$.
\end{lem}
\begin{proof}
The family $f'$ is flat and projective as these properties are preserved under base change. 
Let $y'\in Y'$ be a point and $y\in Y$ be its image. Let $K', K$ be the residue fields of $y',y$ 
respectively, and let $F$ be the fibre of $f$ over $y$. Then the fibre of $f'$ over $y'$ is $F_{K'}$, that is, 
$F$ after base change to $K'$. Now if $F$ is a nodal curve over $K$, then 
$F_{K'}$ is also a nodal curve over $K'$ because $F_{\overline{K}}$ being a connected reduced nodal curve implies 
$F_{\overline{K'}}$ is a connected reduced nodal curve. Moreover, if $F$ is a split nodal curve over $K$, then 
$F_{K'}$ is a split nodal curve over $K'$: indeed, each singular point $\eta'$ of $F_{K'}$ maps to a singular point $\eta$ of $F$, and $\eta$ being $K$-rational implies $\eta'$ is $K'$-rational; also, since the irreducible components of $F$ are geometrically irreducible and smooth, the irreducible components of $F_{K'}$ are geometrically irreducible and smooth.
\end{proof}

A family of split nodal curves can be described locally formally as in the next lemma.

\begin{lem}\label{l-nodal-curves-total-space}
Let $f\colon X\to Y$ be a family of split nodal curves where $Y$ is a variety (over $k$ as usual). 
Let $x\in X$ be a closed point, $y=f(x)$, $A={\mathcal{O}}_{Y,y}$ and  $B={\mathcal{O}}_{X,x}$. Then 
\begin{enumerate}
\item if $f$ is smooth at $x$, then there is an open neighbourhood $U$ of $x$ such that $U\to Y$ factors as the composition of 
an \'etale morphism $U\to \A^1_Y$ followed by the projection $\A^1_Y\to Y$;

\item if $f$ is not smooth at $x$, then there exist $\lambda\in \widehat{A}$ and an isomorphism
$$
\widehat{B}\simeq \widehat{A}[[\alpha,\beta]]/(\alpha\beta-\lambda)
$$
of $\widehat{A}$-algebras where $\widehat{A},\widehat{B}$ are completions and $\alpha,\beta$ 
are independent variables; 

\item if $f$ is not smooth at $x$, then the inverse image of the singular locus ${\rm Sing}(f)$ to $\Spec \widehat{B}$ maps onto the vanishing set of $\lambda$ under the morphism $\Spec \widehat{B}\to \Spec \widehat{A}$.  
\end{enumerate}
\end{lem}
\begin{proof}
(1) 
This follows from  [\ref{Liu}, \S 6.2.2, Corollary 2.11].  (2), (3)
These are proved in [\ref{de-jong-smoothness-semi-stability}, 2.23] (also see [\ref{Liu}, \S 10.3.2, Lemma 3.20]). 
\end{proof}

\subsection{Certain families over toric couples}

\begin{lem}\label{l-fam-over-toric-pairs}
Assume that $(Y,E)$ is a normal toric couple of dimension $d$ and $t_1,\dots,t_d$ are the coordinate functions on the torus $\mathbb{T}_Y=(\A^1\setminus \{0\})^d$. 

\begin{enumerate}
\item Let $\A^1=\Spec k[\alpha]$,  
$
V=Y\times \A^1,
$ 
and $C$ be the inverse image of $E$ plus the vanishing section of $\alpha$. Then $(V,C)$ is a normal toric couple and the projection morphism $(V,C)\to (Y,E)$ is a toric morphism.

\item Let $\lambda\neq 0$ be a character in $t_1,\dots,t_d$ and  $Y^\circ\subset Y$ be the maximal open subset where $\lambda$ is regular. Let $\A^2=\Spec k[\alpha,\beta]$,
$$
X\subset  Y^\circ\times \A^2
$$ 
be the closed subscheme defined by $\Phi:=\alpha\beta-\lambda$, and $D$ be the inverse image of $E$. Then $(X,D)$ is a normal toric couple and the projection morphism $(X,D)\to (Y,E)$ is a toric morphism.
\end{enumerate}
\end{lem}
\begin{proof}

(1) 
This follows from standard toric geometry.
(2)
Here, by a character we mean 
$$
\lambda=t_1^{m_1}\cdots t_{d}^{m_d}
$$ 
where $m_1,\dots,m_d$ are integers (negative integers are allowed). This corresponds to the element $(m_1,\dots,m_d)$ in the character lattice of $Y$. First, we show that $Y^\circ$ is a toric variety. Clearly $Y^\circ$ includes the torus $\mathbb{T}_Y$. Moreover, since $Y$ is normal, $Y\setminus Y^\circ$ is the union of the irreducible components of the divisor $\Div(\lambda)$ on $Y$ with negative coefficients, so  $Y\setminus Y^\circ$ is either empty or a closed subset of pure codimension one. In the first case, $Y^\circ=Y$. In the latter case, $Y\setminus Y^\circ$ is a union of some toric prime divisors, hence its complement is torus-invariant, so it is a toric variety. 
 
Let $g$ be the projection morphism $X\to Y^\circ$. The fibre of $g$ over a closed point $u\in Y^\circ$ is given by the equation $\alpha\beta-\lambda(u)$ on $\A^2$. This fibre is smooth iff $\lambda(u)\neq 0$. Moreover, $g$ is flat: consider the closed subscheme $W$ of $Y^\circ\times \PP^2$ defined by $\alpha\beta-\lambda\gamma^2$ where $\PP^2=\Proj k[\alpha,\beta,\gamma]$; the fibre of $W\to Y^\circ$ over $u$ (closed or not) is given by the equation $\alpha\beta-\lambda(u)\gamma^2$ which is a conic, hence the Hilbert polynomials of these fibres are the same; so we can apply [\ref{Hartshorne}, Chapter III, Theorem 9.9] to deduce that $W\to Y^\circ$ is flat; this in turn implies $g$ is flat.

The general fibres of $g$ are irreducible and smooth (so integral) as they are isomorphic to $\A^1\setminus\{0\}$. 
Thus $X$ is integral [\ref{Liu}, \S 4.3.1, Proposition 3.8], hence it is a variety. 

Next, we will argue that $X$ is normal. Indeed, since $Y^\circ\times \A^2$ is toric and normal, it is Cohen-Macauly, so  
$X$ is Cohen-Macaulay as it is defined by one equation. Therefore, it is enough to show that $X$ is regular in codimension one, by Serre's criterion. Assume not, and let $S$ be a codimension one component of the singular locus of $X$. Then $\dim S=\dim Y^\circ$.
Since $g$ is generically smooth, $S\to Y^\circ$ is not dominant.
 Moreover, since the fibres of $g$ are curves, $S$ dominates a prime divisor $T$ on $Y^\circ$ which can be seen by counting dimensions. 
However, the fibres of $g$ are reduced curves, so $S$ contains at most finitely many points of each fibre of $g$ over each smooth point of $Y^\circ$. Since $Y^\circ$ is normal, $Y^\circ$ is smooth near the generic point of $T$, so the general fibres of $S\to T$ are zero-dimensional. Thus 
$$
d=\dim S=\dim T<\dim Y^\circ=d,
$$ 
a contradiction. Thus we have shown that $X$ is normal.

Now we show that $X$ is a toric variety and that $g\colon X\to Y^\circ$ is a toric morphism. Consider the tori $\mathbb{T}_{Y^\circ}$ and $\mathbb{T}_{\A^2}$. Then $\mathbb{T}_{Y^\circ}\times \mathbb{T}_{\A^2}$ is the torus of $Y^\circ\times \A^2$. Let 
$$
\mathbb{T}_{X}:=X\cap (\mathbb{T}_{Y^\circ}\times \mathbb{T}_{\A^2}).
$$ 
Given each closed point $(u,a,b)\in \mathbb{T}_{X}$, we have $ab-\lambda(u)=0$ but $\lambda(u)\neq 0$ as $\lambda$ is a character and $u\in \mathbb{T}_{Y^\circ}$. Thus $(u,a,b)$ is uniquely determined by $(u,a)$. This shows that $\mathbb{T}_{X}$ is isomorphic (as varieties) to $\mathbb{T}_{Y^\circ}\times \mathbb{T}_{\A^1}$, the torus of dimension $d+1$.
Moreover, since $\lambda$ is a character, $\mathbb{T}_{X}$ is an algebraic subgroup of $\mathbb{T}_{Y^\circ}\times \mathbb{T}_{\A^2}$ and its multiplicative structure inherited from $\mathbb{T}_{Y^\circ}\times \mathbb{T}_{\A^2}$ is compatible with that of $\mathbb{T}_{Y^\circ}\times \mathbb{T}_{\A^1}$. Therefore, the said isomorphism is an isomorphism of tori.  

On the other hand, the action of $\mathbb{T}_{X}$ on itself extends to its closure which is $X$ because $\mathbb{T}_{X}$ acts on $Y^\circ \times \A^2$ as it is a subgroup of $\mathbb{T}_{Y^\circ}\times \mathbb{T}_{\A^2}$. This shows that $X$ is toric. Moreover, the map on tori 
$\mathbb{T}_{X}\to \mathbb{T}_{Y^\circ}$, given by projection, is a group homomorphism of tori. Additionally, $g$ is equivariant with respect to the action of these tori because the projection $Y^\circ\times \A^2\to Y^\circ$ is equivariant with respect to $\mathbb{T}_{Y^\circ}\times \mathbb{T}_{\A^2}\to \mathbb{T}_{Y^\circ}$. 
Thus $g$ is a toric morphism.

Now since $g\colon X\to Y^\circ$ is toric and $Y^\circ$ is a toric open subset of $Y$, the induced morphism $f\colon X\to Y$ is toric. 
For each closed point $u\in \mathbb{T}_{Y}$, every point $(u,a,b)$ in $f^{-1}\{u\}$ is contained in $\mathbb{T}_{X}$ as $ab=\lambda(u)\neq 0$. Thus $\mathbb{T}_{X}=f^{-1}\mathbb{T}_{Y}$ and $D=f^{-1}E$ is the complement of $\mathbb{T}_X$. 
\end{proof}

\subsection{Families of nodal curves over toroidal couples}

In this subsection we study families of split nodal curves over a base that has a toroidal structure. We aim to get a toroidal structure on the total space.

\begin{lem}\label{l-Cartier-near-point-toric}
Assume $(Y,E)$ is a normal toric couple, $y\in Y$ is a closed point, and $H\ge 0$ is a Weil divisor with $\Supp H\subseteq E$. If $H$ is Cartier near $y$, then there is a character $\gamma$ such that $H=\Div(\gamma)$ on some torus-invariant open neighbourhood of $y$. 
\end{lem}
\begin{proof}
Fix an open neighbourhood $U\subset Y$ of $y$ on which $H$ is Cartier. Each closed point of the torus $\mathbb{T}_Y$ gives an automorphism $g\colon Y\to Y$. Then the union of all the $g(U)$ is torus-invariant. Moreover, since $\Supp H\subseteq E$, we see that $H$ is torus-invariant, so $H=g^*H$. Thus $H$ is Cartier on each $g(U)$, hence is Cartier on their union. Thus replacing $Y$ with the union, we can assume $H$ is Cartier everywhere. Now we can apply  [\ref{Cox-etal}, Proposition 4.2.2] to some torus-invariant affine open neighbourhood of $y$.  
\end{proof}

\begin{prop}\label{p-fam-nod-curves-over-toric}
Let $(X,D)$ and $(Y,E)$ be couples, and $f\colon X\to Y$ be a family of split nodal curves.
Assume that 
\begin{itemize}
\item $(Y,E)$ is toroidal, 

\item $f$ is smooth over $Y\setminus E$, 

\item the horizontal components of $D$ are disjoint sections of $f$ contained in the smooth locus of $f$, and 

\item the vertical part of $D$ is equal to $f^{-1}E$.
\end{itemize}
Then $(X,D)$ is a toroidal couple and $(X,D)\to (Y,E)$ is a toroidal morphism.
\end{prop}
\begin{proof} 
\emph{Step 1}. 
Let $x\in X$ be a closed point and $y\in Y$ be its image. Since $(Y,E)$ is toroidal, 
there is a local toric model $(Y',E'),y'$ of $(Y,E),y$ where $Y'$ is normal. Since 
$(Y',E')$ is a toric couple and $Y'$ is normal, each component of $E'$ is normal.
Let $A={\mathcal{O}}_{Y,y}$, $B={\mathcal{O}}_{X,x}$, and $A'={\mathcal{O}}_{Y',y'}$. In the following steps, we will construct a local toric model $(X',D'),x'$ of $(X,D),x$, over $(Y',E'),y'$.

\emph{Step 2}.
First, assume $f$ is not smooth at $x$, hence $x$ is a node on the fibre over $y$. 
Then by Lemma \ref{l-nodal-curves-total-space}, 
$$
\widehat{B}\simeq \widehat{A}[[\alpha,\beta]]/(\alpha\beta-\lambda)
$$
for some $\lambda\in \widehat{A}$. Let $\widehat{H}$ be the effective Cartier divisor on 
$\widehat{Y}=\Spec \widehat{A}$ defined by $\lambda$. By Lemma \ref{l-nodal-curves-total-space}, the inverse image of the singular locus ${\rm Sing}(f)$ to $\widehat{X}=\Spec \widehat{B}$ maps onto $\Supp \widehat{H}$ under the morphism $\widehat{X}\to \widehat{Y}$. Moreover, since $f$ is smooth over $Y\setminus E$, we deduce that 
${\rm Sing}(f)$ maps into $E$, hence $\Supp \widehat{H}\subseteq \widehat{E}$  
where $\widehat{E}$ is the divisor on $\widehat{Y}$ determined by $E$. In particular, $\lambda\neq 0$ because 
$\widehat{E}$ is a proper subset of $\widehat{Y}$. 

On $Y'$, write $E'=\sum E_i'$ where $E_i'$ are the irreducible components. Since $Y'$ is toric and normal, $E_i'$ is normal. Consider $\widehat{E}'=\sum \widehat{E}_i'$ on 
 $\widehat{Y}'=\Spec \widehat{A'}$. Since $E_i'$ is normal, $\widehat{E}_i'$ is normal (see the proof of Lemma \ref{l-toroidal-couple-lc}). Thus $\widehat{E}_i'$ is a prime divisor if $y'\in E_i'$ and $\widehat{E}_i'=0$ otherwise.
 
\emph{Step 3}. 
Now since $(Y,E),y$ and $(Y',E'),y'$ are formally isomorphic, there is an isomorphism $\widehat{Y}\to \widehat{Y}'$ mapping $\widehat{E}$ to $\widehat{E}'$.  So 
$\widehat{H}$ corresponds to an effective Cartier divisor $\widehat{H}'$ on 
 $\widehat{Y}'$ such that $\Supp \widehat{H}'\subseteq \widehat{E}'$. From this we deduce that 
 $\widehat{H}'=\sum l_i\widehat{E}_i'$ for certain non-negative integers $l_i$ because for each $i$, $\widehat{E}_i'$ is a prime divisor or is zero. Then 
 $\widehat{H}'$ is the divisor associated to $H':= \sum l_i{E}_i'$. By Lemma \ref{l-formally-cartier}, 
 $H'$ is Cartier near $y'$. Applying Lemma \ref{l-Cartier-near-point-toric}, we see that, near $y'$,
$H'=\Div(\gamma)$ for some character $\gamma$ on $Y'$. Then we can assume $H'=\Div(\gamma)$ holds on the regular locus $Y'^\circ$ of $\gamma$.

Consider $\lambda,\gamma$ as elements of $\widehat{A'}$. 
Since both $\lambda,\gamma$ define the same Cartier divisor $\widehat{H}'$ on $\widehat{Y}'$,
we have $\lambda=\gamma \rho$ in $\widehat{A'}$  for some invertible element $\rho\in \widehat{A'}$.
Replacing $\alpha$ with $\alpha/\rho$, we may assume $\lambda=\gamma$ in $\widehat{A'}$. From now on we will use $\lambda$ instead of $\gamma$.

\emph{Step 4}.
Let $X'$ be the closed subscheme of  $Y'^\circ\times \A^2$ defined by the equation $\alpha\beta-\lambda$ where $\alpha,\beta$ are considered as coordinate variables on $\A^2$.
Let $f'\colon X'\to Y'$ be the induced morphism, and $D'$ be the inverse image of $E'$. By Lemma \ref{l-fam-over-toric-pairs}, $(X',D')\to (Y',E')$ is a toric morphism of normal toric couples. 
The general fibres of $f'$ are isomorphic to $\A^1\setminus\{0\}$.

Since the fibre of $f$ over $y$ is singular by assumption in Step 2, $\lambda$ vanishes at the closed point of $\widehat{Y}\simeq \widehat{Y}'$, so it also 
vanishes at $y'$, hence the fibre of $f'$ over $y'$ is also singular. Let $x'\in X'$ be the node of the fibre over $y'$. Then $x'=(y',(0,0))$ and
$$
\widehat{\mathcal{O}}_{X',x'}\simeq \widehat{\mathcal{O}}_{Y',y'}[[\alpha,\beta]]/(\alpha\beta-\lambda) 
\simeq \widehat{\mathcal{O}}_{Y,y}[[\alpha,\beta]]/(\alpha\beta-\lambda)
\simeq \widehat{\mathcal{O}}_{X,x}.
$$
Moreover, the ideal of $D'$ in $\widehat{\mathcal{O}}_{X',x'}$ corresponds to the ideal of $D$ in $\widehat{\mathcal{O}}_{X,x}$ because $D'=f'^{-1}E'$ by definition and $D=f^{-1}E$ near $x$ by assumption (recall that no horizontal component of $D$ passes through $x$ because such components are contained in the smooth locus of $f$, and the vertical part of $D$ is $f^{-1}E$,  by assumption), and because the ideals of $E'$ and $E$ correspond via the given  isomorphism $\widehat{\mathcal{O}}_{Y',y'}\simeq \widehat{\mathcal{O}}_{Y,y}$.
Therefore,
$$
(X',D'),x' \to (Y',E'),y'
$$ 
is a local toric model of 
$$
(X,D),x \to (Y,E),y.
$$

\emph{Step 5}.
Now assume $f$ is smooth at $x$. Then by Lemma \ref{l-nodal-curves-total-space}, 
there is a neighbourhood $U$ of $x$ such that the induced morphism $U\to Y$ factors as an 
\'etale morphism $U\to \A^1_Y$ followed by the projection $\A^1_Y\to Y$. Let
$$
X'=\A^1_{Y'}=Y'\times \A^1,
$$ 
let $f'\colon X'\to Y'$ be the projection, and let $D'\subset X'$ be the inverse image of $E'$ plus the section defined by the vanishing of $\alpha$ where $\A^1=\Spec k[\alpha]$ in both $\A^1_Y=Y\times \A^1$ and $\A^1_{Y'}=Y'\times \A^1$. Then by Lemma \ref{l-fam-over-toric-pairs},
$(X',D')\to (Y',E')$ is a toric morphism of normal toric couples. 

Assume that $x$ does not belong to any horizontal component of $D$. Then we can choose the map $U\to \A^1_Y=Y\times \A^1$ so that $x$ maps to $(y,1)$. Let $x'=(y',1)\in X'$. Then 
$$
\widehat{\mathcal{O}}_{X',x'}\simeq \widehat{\mathcal{O}}_{X,x}
$$ 
and the ideal of $D'$ in $\widehat{\mathcal{O}}_{X',x'}$ corresponds to the ideal of $D$ 
in $\widehat{\mathcal{O}}_{X,x}$ because $D=f^{-1}E$ near $x$ and $D'=f^{-1}E'$ near $x'$. 
Therefore,  
$$
(X',D'),x' \to (Y',E'),y'
$$  
is a local toric model of 
$$
(X,D),x \to (Y,E),y.
$$ 

Assume $x$ belongs to a horizontal component $T$
of $D$. By assumption, $T$ is unique (containing $x$) and $T$ is a section of $f$ contained in the smooth locus of $f$. 
Then $U\cap T$ is mapped isomorphically onto an open subset of $Y$. Replacing $Y$ with this open subset, we can assume that $U\cap T$ is mapped isomorphically onto $Y$, hence in particular, $U\cap T=T$. 
Now $U\to \A^1_Y$ maps $T$ onto a section of $\A^1_Y\to Y$. Moreover, we can assume that this section of $\A^1_Y\to Y$ is the vanishing section of $\alpha$: indeed, we can assume $Y$ is affine, so each section of $\A^1_Y$ corresponds to a surjection $k[Y][\alpha]\to k[Y]$ 
which is the identity on $k[Y]$; this surjection is determined by sending $\alpha$ to an element $\sigma$; so the kernel of the map is generated by $\alpha-\sigma$, and the ideal of 
the section is generated by $\alpha-\sigma$; changing the variable $\alpha$ to $\alpha+\sigma$ on $Y\times \A^1$, we can assume the ideal of the section is generated by $\alpha$.

By the previous paragraph, $U\to \A^1_Y$ maps $x$ to $(y,0)$. In particular, $\widehat{\mathcal{O}}_{X,x}$ is isomorphic to $\widehat{\mathcal{O}}_{Y,y}[[\alpha]]$ (cf. [\ref{Matsumura}, Exercise 8.6]). 
Now let $x'=(y',0)$. Then again 
$$
\widehat{\mathcal{O}}_{X',x'}\simeq \widehat{\mathcal{O}}_{Y',y'}[[\alpha]]\simeq  
\widehat{\mathcal{O}}_{Y,y}[[\alpha]]\simeq  \widehat{\mathcal{O}}_{X,x}
$$ 
 and we can check that the ideal of $\widehat{D'}$ corresponds to the ideal of $\widehat{D}$ because $D$ on $X$ corresponds to the union of the inverse image of $E$ and the section of $\alpha$ on $Y\times \A^1$ which in turn corresponds to $D'$ on $X'$ which is the union of the inverse image of $E'$ and the section of $\alpha$. Note that we are also implicitly using the fact that the section of $\alpha$ on $Y\times \A^1$ is normal, hence its inverse image to $U$ is normal, so it coincides with $T$ near $x$.  
To summarise, we have again shown that 
$$
(X',D'),x' \to (Y',E'),y'
$$  
is a local toric model of 
$$
(X,D),x \to (Y,E),y.
$$  
\end{proof}

\subsection{Good towers of families of nodal curves}\label{ss-tower-family-nodal-curves}
We introduce certain towers of couples as in \ref{ss-tower-couples} but with stronger properties.

(1) A \emph{good tower of families of (split)
nodal curves} 
$$
(V_d,C_d)\to (V_{d-1},C_{d-1})\to \cdots \to (V_1,C_1)
$$ 
consists of couples $(V_i,C_i)$ and morphisms $g_i\colon V_i\to V_{i-1}$ such that 
\begin{itemize}
\item $g_i$ is a family of (split) nodal curves,
\item $g_i$ is smooth over $V_{i-1}\setminus C_{i-1}$, 
\item the horizontal$/V_{i-1}$ components of $C_i$ are disjoint sections of 
$g_i$ contained in the smooth locus of $g_i$, and 
\item the vertical$/V_{i-1}$ part of $C_i$ is equal to $g_i^{-1}C_{i-1}$.
\end{itemize}
Note that we are implicitly assuming that $g_i$ are flat, surjective, and projective. Also the tower above is a tower of couples as defined in \ref{ss-tower-couples}.

(2) Given a tower as in (1), we show that the 
fibre $F_i$ of $V_i\to V_1$ over any closed point $v\in V_1\setminus C_1$ is integral and not contained in $C_i$. Indeed, by definition, $F_2$ is smooth and being a nodal curve it is connected, hence it is irreducible. Also $F_2$ is not contained in $C_2$ because the vertical part of $C_2$ is $g_2^{-1}C_1$ and the horizontal part of $C_2$ is a disjoint union of sections. Inductively, we can assume $F_{i-1}$ is integral and that it is not contained in $C_{i-1}$. 
Since $F_{i-1}$ is not contained in $C_{i-1}$, the general fibres of $F_i\to F_{i-1}$ are smooth and irreducible as $g_i$ is smooth over $V_{i-1}\setminus C_{i-1}$. Therefore, $F_i$ is integral by [\ref{Liu}, \S 4.3.1, Proposition 3.8] as $g_i$ is flat. Moreover, the general fibres of $F_i\to F_{i-1}$ are not contained in $C_i$, so $F_i$ is not contained in $C_i$.

(3) Given a tower as in (1), let $(X_1,E_1)$ be a couple and $X_1\to V_1$ be a morphism whose image is not contained in $C_1$ (but we are not assuming $(X_1,E_1)\to (V_1,C_1)$ to be a morphism of couples). Also assume that $E_1$ and each $C_i$ is the support of an effective Cartier divisor.  Let $X_i=X_1\times_{V_1}V_i$ and let $D_i$ be the union of the inverse images of $E_1$ and $C_i$. Then we show that the induced tower 
$$
(X_d,D_d)\to (X_{d-1},D_{d-1})\to \cdots \to (X_1,D_1)
$$ 
is a good tower of (split) nodal curves. First, each $h_i\colon X_i\to X_{i-1}$ is a family of (split) nodal curves, by Lemma \ref{l-nodal-curves-base-change}, which is smooth over $X_{i-1}\setminus D_{i-1}$. Second, since $X_1$ is not mapped into $C_1$, the general fibres of $X_i\to X_1$ are integral and not contained in $D_i$, by (2). Thus $X_i$ is a variety as $X_i\to X_1$ is flat, and so $(X_i,D_i)$ is a couple.

On the other hand, the horizontal$/X_{i-1}$ part of $D_i$ is the inverse image of the horizontal$/V_{i-1}$ part of $C_i$, so its components are disjoint sections of $h_i$  contained in the smooth locus of $h_i$. Moreover, the vertical$/X_{i-1}$ part of $D_i$ is equal to $h_i^{-1}D_{i-1}$: indeed, the inverse 
image of $D_{i-1}$ is contained in the vertical$/X_{i-1}$ part of $D_i$; conversely, if $L$ is a 
vertical$/X_{i-1}$ component of $D_i$, then either $L$ is a component of the inverse image of $E_1$ in which case $L$ is mapped into $D_{i-1}$, or $L$ is mapped into the vertical$/V_{i-1}$ part of $C_i$ (as the inverse image of the horizontal part of $C_i$ is a disjoint union of sections of $h_i$) in which case $L$ is mapped into $C_{i-1}$, hence again $L$ maps into $D_{i-1}$.

\subsection{Altering a fibration into a good tower of families of nodal curves}\label{ss-altering-family-into-nodals}

\begin{prop}\label{p-tower-nod-curve-fix-fib}
Assume  $(V,C)$ is a couple and $f\colon V\to T$ is a surjective projective morphism. 
Then there exists a commutative diagram of couples 
$$
 \xymatrix{
 (V_{d},C_d) \ar[d]\ar[r]^{\nu} &  (V,C) \ar[dd]^f \\
  \vdots \ar[d] & \\
  (V_1,C_1) \ar[r] & T,
  } 
$$  
where
\begin{itemize}
\item the left hand side is a good tower of families of split nodal curves,

\item $\nu\colon V_d\to V$ and $V_1\to T$ are alterations,

\item  $(V_i,C_i)$ are toroidal and $(V_1,C_1)$ is log smooth, 

\item $C_i$ is the support of some effective Cartier divisor, and 

\item the induced morphism 
$$
\nu|_{V_d\setminus C_d }\colon V_d\setminus C_d \to V\setminus C
$$ 
is quasi-finite.
\end{itemize}
\end{prop}
\begin{proof} 
\emph{Step 1}.
We apply induction on 
$$
d:=\dim V-\dim T+1.
$$ 
If $d=1$, then $f$ is generically finite, so we can take $V_1\to V$ to 
be a log resolution and  $C_1$ be the birational transform of $C$ union the exceptional divisors. 
We then assume $d\ge 2$.
By Lemma \ref{l-produce-fib-rel-dim-one},  
we can find a resolution $V'\to V$ and a contraction $V'\to W'/T$ of relative dimension one. Let $C'\subset V'$ be 
an effective Cartier divisor whose support contains the birational transform of $C$ and such that 
$V'\to V$ restricted to $V'\setminus \Supp C'$ is an isomorphism onto its image. 
Let $G'$ be an effective Cartier divisor on 
$W'$ so that $W'\setminus \Supp G'$ is smooth and $V'\to W'$ is smooth over $W'\setminus \Supp G'$.
Replacing $C',G'$, we can assume the support of the vertical$/W'$ part of $C'$ is equal to the support of the pullback of $G'$.

\emph{Step 2}.
By [\ref{de-jong-nodal-family}, Theorem 2.4], there exist alterations 
$V''\to V'$ and $W''\to W'$ and an induced morphism $V''\to W''$ which is a family of nodal curves with smooth generic fibre.  
 Applying [\ref{de-jong-smoothness-semi-stability}, Theorem 5.8], 
 we can assume $V''\to W''$ is a family of split nodal curves.
Let $C''\subset V''$ be the pullback of $C'$ and let $G''\subset W''$ be the pullback of $G'$. 
We can moreover assume that the support of the horizontal$/W''$ part of ${C}''$ is a 
disjoint union of sections of $V''\to W''$ contained in the smooth locus of $V''\to W''$. 
After replacing $G'$ and the vertical$/W'$ part of $C'$, and replacing $G''$ and the vertical$/W''$ part of $C''$ accordingly, we can assume 
 $W''\setminus \Supp G''$ is smooth, $V''\to W''$ is smooth over $W''\setminus \Supp G''$, and that the support of the pullback of $G''$ is the support of the vertical$/W''$ part of $C''$. 
 In addition, we can assume   $W''\to W'$ is \'etale over $W'\setminus \Supp G'$.
 
\emph{Step 3}. 
Applying induction to the couple $(W'', \Supp G'')$ and the morphism $W''\to T$, there exists a commutative diagram 
$$
 \xymatrix{
 (V_{d-1},C_{d-1}) \ar[d]\ar[r]^-\mu &  (W'', \Supp G'') \ar[dd] \\
  \vdots \ar[d] & \\
  (V_1,C_1) \ar[r] & T,
  } 
$$ 
where 
\begin{itemize}
\item the left hand side is a good tower of families of split nodal curves,

\item $(V_i,C_i)$ are toroidal and $(V_1,C_1)$ is log smooth, 

\item $\mu\colon V_{d-1}\to W''$ and $V_1\to T$ are alterations, 

\item $C_i$ is the support of some effective Cartier divisor, and 

\item  $\mu|_{V_{d-1}\setminus C_{d-1}}$ gives the morphism $V_{d-1}\setminus C_{d-1}\to W''\setminus \Supp G''$ which is quasi-finite (in particular, $C_{d-1}$ contains the support of the pullback of $G''$).
\end{itemize}

\emph{Step 4}.
Let 
$$
V_d:=V_{d-1}\times_{W''}V''.
$$ 
Then, by Lemma \ref{l-nodal-curves-base-change}, the induced morphism  $V_d\to V_{d-1}$ is a family of split nodal curves with smooth general fibres.
Let $C_d\subset V_d$ be the support of the pullback of $C''$ union the support of 
the pullback of $C_{d-1}$. 

We will argue that $(V_d,C_d)$ is a couple. We need to show that $V_d$ is a variety and that $C_d$ is a reduced divisor. The latter follows from the former as 
$C_d$ is the support of some effective Cartier divisor. By construction, $V_{d-1}$ is a variety and $V_d\to V_{d-1}$ is flat
with integral general fibres. Then $V_d$ is integral hence a variety, and so $(V_d,C_d)$ is a couple. 

By construction, 
the horizontal$/V_{d-1}$ components of $C_d$ are 
disjoint sections of  $V_d\to V_{d-1}$ contained in the smooth locus of $V_d\to V_{d-1}$, and the 
vertical$/V_{d-1}$ part of $C_d$ coincides with the inverse image of $C_{d-1}$. In addition, 
$V_d\to V_{d-1}$ is smooth over $V_{d-1}\setminus C_{d-1}$ as $C_{d-1}$ contains the support of the pullback of $G''$. 
Moreover, by construction, the induced morphism $\nu\colon V_d\to V$ is an alteration and 
$\nu(V_d\setminus C_d)\subseteq V\setminus C$. By Proposition \ref{p-fam-nod-curves-over-toric}, $(V_d,C_d)$ is a toroidal couple and $(V_d,C_d)\to (V_{d-1},C_{d-1})$ is a toroidal morphism.

\emph{Step 5}.
We show that we can run the above arguments so that 
$\nu|_{{V_d\setminus C_d}}$ is quasi-finite. 
The morphism $V''\to V'$ factors as a birational contraction $V''\to S$ followed by a finite morphism $S\to V'$. 
In particular, $V''\to V'$ is finite over the complement of a codimension $2$ closed subset $Q'$ of $V'$.   
Since $V'\to W'$ has relative dimension one, $Q'$ is vertical$/W'$, hence at this point we can add to $G'$ (and accordingly to $C'$) 
so that $Q'\subset \Supp C'$. Thus 
$$
V''\setminus \Supp C''\to V'\setminus \Supp C'
$$ 
is 
finite which in turn implies 
$$
V''\setminus \Supp C''\to V\setminus  C
$$ 
is finite onto an open subset of $V\setminus  C$. 

On the other hand,
$$
V_{d-1}\setminus C_{d-1}\to W''\setminus \Supp G''
$$ 
is quasi-finite, hence $V_{d}\setminus N_{d}\to V''$ is quasi-finite where $N_d$ is the inverse image of $C_{d-1}$. But 
$C_d$ contains both $N_d$ and the pullback of $C''$, hence the induced morphism 
$$
V_d\setminus C_d\to V''\setminus \Supp C''
$$ 
is quasi-finite. Therefore,  
$
V_d\setminus C_d\to V\setminus C
$
is quasi-finite.
\end{proof}

\subsection{Bounded toroidalisation of fibrations over curves}

\begin{proof}[Proof of Theorem \ref{t-bnd-torification}]
\emph{Step 1.}
We apply induction on $d$. The case $d=1$ is trivial as we can take $X'=Z'$ to be the normalisation of $X$, $\mu\colon Z'\to Z$ and $\pi\colon X'\to X$ the induced morphisms, and $D'=E'$ be the union of the inverse images of $D$ and $z$, hence $(X',D')\to (Z',E')$ is the identity morphism which is toroidal. Thus we assume $d\ge 2$. 
Let $\mathcal{P}$ be a set of  couples $(X/Z,D)$ satisfying the assumptions of the theorem, e.g. initially 
we can take $\mathcal{P}$ to be the set of all possible $(X/Z,D)$. Removing the vertical$/Z$ 
components of $D$ we can assume that every component of $D$ is horizontal$/Z$: note that in the end we will have a morphism $X'\setminus D'\to X\setminus (D+f^*z)$, so removing the vertical part of $D$ does not cause problems.  

By Lemma \ref{l-univ-family}, there exist finitely many projective morphisms ${V}^i\to T^i$ of varieties 
and reduced divisors $C^i\subset V^i$ depending only on $d,r$ such that for each 
$(X/Z,D)\in\mathcal{P}$, there is an $i$ and a morphism $Z\to T^i$ 
such that $X=Z\times_{T^i}V^i$ and $D=Z\times_{T^i}C^i$. 
Replacing $\mathcal{P}$,  we can fix $i$, hence write $V,T,C$ instead of $V^i,T^i,C^i$.

\emph{Step 2.}
By Proposition \ref{p-tower-nod-curve-fix-fib}, we can alter $(V,C)\to T$ into a good tower of 
families of split nodal curves  
$$
 \xymatrix{
 (V_{d},C_d) \ar[d]\ar[r]^{\nu} &  (V,C) \ar[dd] \\
  \vdots \ar[d] & \\
  (V_1,C_1) \ar[r] & T
  } 
$$  
where 
\begin{itemize}
\item the left hand side is a tower of families of split nodal curves,

\item $\nu\colon V_d\to V$ and $V_1\to T$ are alterations,

\item  $(V_i,C_i)$ are toroidal and $(V_1,C_1)$ is log smooth, 

\item $C_i$ is the support of some effective Cartier divisor, and

\item $\nu|_{V_d\setminus C_d}\colon V_d\setminus C_d \to V\setminus C$ is quasi-finite.
\end{itemize} 
Let $S\subset T$ be a proper closed subset 
such that $V_1\to T$ is a finite \'etale morphism over $T\setminus S$ and that $C_1$ is mapped into $S$. We can remove those  $(X/Z,D)\in\mathcal{P}$ for which 
the image of $Z\to T$ is contained in $S$ because for such couples 
we can replace $T$ by a component of $S$ and replace $(V,C)$ accordingly, and do induction on dimension of $T$. So from now on we assume 
the image of $Z\to T$ intersects $T\setminus S$.

\emph{Step 3.}
Let $X_1'$ be the normalisation of a component of $Z\times_TV_1$ dominating $Z$. By Step 2, the image of $X_1'\to V_1$ is not contained in $C_1$. 
Moreover, the induced morphism $\rho\colon X_1'\to Z$ is finite of degree not exceeding the degree of $V_1\to T$. 
Let $X_i'=X_1'\times_{V_1}V_i$. 

Pick a closed subset $Q\subset V$ so that $\nu(C_d)\subseteq Q$ and $V_d\to V$ is \'etale over $V\setminus Q$. Let $\mathcal{Q}\subseteq \mathcal{P}$ be the set of those couples $(X/Z,D)$ such that the image of $X\to V$ is contained in $Q$. To deal with these couples, by Lemma \ref{l-fibred-product}, we can replace the family $V\to T$ with finitely many new families $V^j\to T^j$ where $\dim V^j<\dim V$, hence we can apply induction on dimension of $V$. Thus we remove the elements of $\mathcal{Q}$ from $\mathcal{P}$, hence we assume that for every couple $(X/Z,D)$ in $\mathcal{P}$, the image of $X\to V$ is not contained in $Q$.
Then we can assume that $X_d'$ is not mapped into $C_d$ and that $V_d\to V$ is \'etale over the generic point of the image of $X\to V$. In particular, this implies that $X_i'$ is not mapped into $C_i$ by $X_i'\to V_i$ because the inverse image of $C_i$ to $V_d$ is contained in $C_d$, and moreover $\deg (X_d'\to X)\le \deg(V_d\to V)$. 

Now let $B_i'\subset X_i'$ be the inverse image of $C_i$ under $X_i'\to V_i$ with reduced structure. 
Note that since $C_i$ is the support of 
some effective Cartier divisor, $B_i'$ is a divisor. 
Let $D_i'$ be the union of $B_i'$ and the support of the fibres of $X_i'\to X_1'$ over the points in $\rho^{-1}\{z\}$ (that is, union with the support of the fibre of $X_i'\to Z$ over $z$).

\emph{Step 4.}
By \ref{ss-tower-family-nodal-curves}(3), the induced tower  
$$
(X_d',D_d')\to \cdots \to (X_1',D_1')
$$
is a good tower of families of split nodal curves. Therefore, since $(X_1',D_1')$ is log smooth, applying Proposition \ref{p-fam-nod-curves-over-toric}, we deduce that 
$(X_i',D_i')$ is toroidal and $(X_i',D_i')\to (X_{i-1}',D_{i-1}')$ is a toroidal morphism for each $i$.

Since $D$ is the inverse image of $C$ under $X\to V$ and since $C_d$ contains the inverse image of 
$C$ under $V_d\to V$, we deduce $D_d'$ contains the inverse image of $D$ under $X_d'\to X$. Thus we get 
a morphism $X_d'\setminus D_d'\to X\setminus D$ and $(X_d',D_d')\to (X,D)$ is a morphism of couples. 

We claim that $X_d'\setminus D_d'\to X\setminus D$ is quasi-finite. Assume not, say this morphism contracts a curve $\Gamma'$. 
First, assume $X_1'\to V_1$ is not constant, which is then a quasi-finite morphism (not necessarily surjective). Then $X_d'\to V_d$ is also quasi-finite,  hence 
$\Gamma' $ is mapped to a curve $\Gamma$ in $V_d$ which is contracted by $V_d\to V$. 
Then $\Gamma\subset C_d$ by Step 2, so $\Gamma'\subset D_d'$, a contradiction. Now assume $X_1'\to V_1$ is constant, which means $Z\to T$ is also constant. In this case, $X=F\times Z$ for some fibre $F$ of $V\to T$ and $X_d'=G\times X_1'$ for some fibre $G$ of $V_d\to V_1$. Since $\Gamma'$ is contracted by $X_d'\to X$, $\Gamma'$ is contained in a fibre of $X_d'\to X_1'$, so $\Gamma'$ is mapped to a curve $\tilde{\Gamma}\subseteq G\subseteq V_d$ which is in turn contracted by $V_d\to V$. But then $\tilde{\Gamma}\subset C_d$, so $\Gamma'$ is contained in $D_d'$, a contradiction.

\emph{Step 5.}
Now put 
$$
(X',D'):=(X_d',D_d') \ \ \mbox{and} \ \ (Z',E'):=(X_1',D_1').
$$ 
There is an effective Cartier divisor $G_d$ on $V_d$ whose support is $C_d$. Pick a very ample$/V_1$ divisor $A_d$ on $V_d$ so that $A_d-G_d$ is ample over $V_1$. Let $A'$ on $X'$ be the pullback of $A_d$ and $G'$ be the pullback of $G_d$ plus the pullback of $E'$. Then $A'-G'$ is ample over $Z'$ and $D'=\Supp G'$. 
 Moreover, we can assume that 
$$
\deg_{A'/Z'}A'=\deg_{A_d/V_1}A_d
$$
and 
$$
\deg_{A'/Z'}D'\le \deg_{A'/Z'}G'= \deg_{A_d/V_1}G_d\le \deg_{A_d/V_1}A_d.
$$
Therefore, we can choose $r'$ 
depending only on $d,r$ such that  
$$
\deg_{A'/Z'}A'\le r' ~~~ \mbox{and}~~~ \deg_{A'/Z'}D'\le r'.
$$ 
Also note that, by construction, $D'$ contains the fibre of $X'\to Z$ over $z$ as $E'$ contains the inverse image of $z$ under $Z'=X_1'\to Z$. Moreover, replacing $r'$ we can assume 
$$
\deg (Z'\to Z)\le \deg (V_1\to T)\le r'.
$$
Also we can assume that 
$$
\deg (X'\to X) \le \deg(V_d\to V)\le r'
$$
where the first inequality follows from Step 3.

Finally, we can assume that $A$ on $X$ is the pullback of some very ample$/T$ divisor $H$ on $V$ (by the proof of Lemma \ref{l-univ-family}), so we can choose $A_d$ so that $A_d-H|_{V_d}$ is ample$/V_1$, hence $A'-\pi^*A$ is ample$/Z'$, where $\pi$ denotes $X'\to X$.
\end{proof}

\section{\bf Toric models of toroidal fibrations}

In this section we define special toric towers, study their geometry, and relate them to good towers of families of split nodal curves. 

\subsection{Special toric towers}\label{ss-special-toric-tower}\

(1)
A \emph{toric tower}
$$
(V_d,C_d)\to (V_{d-1},C_{d-1})\to \cdots \to (V_1,C_1)
$$
consists of toric couples $(V_i,C_i)$ and dominant toric morphisms $V_i\to V_{i-1}$ (but not necessarily projective). Note that since the torus of $V_i$ is mapped into the torus of $V_{i-1}$, the inverse image of $C_{i-1}$ is contained in $C_i$, so the above tower is a tower of couples as in \ref{ss-tower-couples}.  

(2)
We say a toric tower as in (1) is \emph{special} if it is defined as follows: 

\begin{itemize}
\item $V_1=\A^p=\Spec k[t_1,\dots,t_p]$ and $C_1$ is the vanishing set of $t_1\cdots t_p$, for some $p$,

\item $(V_i,C_i)$ and $\Phi_i$ are defined inductively as follows; assuming we have already defined 
$(V_{j},C_{j})$ and $\Phi_{j}$ for $j\le i-1$, we have either  
{\begin{itemize}
\item[$(a)$]  $\Phi_i=0$ and  
$$
V_i=V_{i-1}\times \A^1
$$ 
and $C_i$ is the inverse image of $C_{i-1}$ plus the vanishing section of $\alpha_i$, where $\alpha_i$ is a new variable on $\A^1$, 
 or 
\item[$(b)$] $\Phi_i=\alpha_i\alpha_i'-\lambda_i$ 
where $\lambda_i$ is a non-zero character in the variables $\alpha_2,\dots, \alpha_{i-1}$, $t_1,\dots,t_p$ and  
$$
V_i\subset V_{i-1}^\circ\times \A^2
$$ 
is the closed subscheme defined by $\Phi_i$, where $V_{i-1}^\circ\subset V_{i-1}$ is the maximal open subset where $\lambda_i$ is regular, $\alpha_i,\alpha_i'$ are new variables on $\A^2$, and $C_i$ is the inverse image of $C_{i-1}$,
\end{itemize}
}
\item $V_i\to V_{i-1}$ are given by projection.
\end{itemize}

By Lemma \ref{l-fam-over-toric-pairs} and induction on $i$, $(V_i,C_i)$ are normal toric couples and $(V_i,C_i)\to (V_{i-1},C_{i-1})$ are toric morphisms. In particular, this means $C_i$ are reduced divisors, so they have no non-divisorial components (we can also use flatness of $V_i\to V_{i-1}$ and reducedness of its fibres and apply [\ref{Liu}, \S 4.3.1, Proposition 3.8] to even deduce that $C_i$ are reduced Cartier divisors). In both cases, $\mathbb{T}_{V_i}\simeq \mathbb{T}_{V_{i-1}}\times \mathbb{T}_{\A^1}$: this is obvious in case $(a)$; in case $(b)$, we use the fact that $\alpha_i'=\frac{\lambda_i}{\alpha_i}$ on the locus where $\alpha_i$ does not vanish. Also, the isomorphism is an isomorphism of tori as $\lambda_i$ is a character. For more details, see the proof of Lemma \ref{l-fam-over-toric-pairs}. In particular, $\alpha_2,\dots, \alpha_i, t_1,\dots,t_p$ are the coordinate functions on the torus $\mathbb{T}_{V_i}$ of dimension $i-1+p$.
Moreover, $V_i\to V_{i-1}$ is flat with smooth integral fibres over $V_{i-1}\setminus C_{i-1}$.

(3)
 Given a special toric tower as in (2), let $F_i$ be the fibre of $V_i\to V_1$ over a closed point $v\in V_1\setminus C_1$. 
We claim that $F_i$ is integral and not contained in $C_i$, for each $i$. In case $(a)$, $F_i=F_{i-1}\times \A^1$, so $F_i$ is integral by induction on $i$. In case $(b)$, $F_i\to F_{i-1}$ is flat with smooth integral general fibres as $\lambda_i$ is a character not vanishing at any point of $F_{i-1}\setminus C_{i-1}$.
Thus $F_i$ is integral [\ref{Liu}, \S 4.3.1, Proposition 3.8].

On the other hand, $F_i$ is not contained in $C_i$: indeed, pick any closed point $w\in F_{i-1}\setminus C_{i-1}$; then this point belongs to the torus of $V_{i-1}$, hence belongs to $V_{i-1}^\circ\cap F_{i-1}$; then the fibre of $V_i\to V_{i-1}$ over $w$ is not contained in $C_i$ in either cases $(a)$,$(b)$; but this fibre  is the same as the fibre of $F_i\to F_{i-1}$ over $w$, so $F_i$ is not contained in $C_i$. 

(4)
Given a special toric tower as in (2), we claim that there is a natural congruent birational map 
$$
({V}_d, C_d) \bir (P=\PP^{d-1}_{V_1},G)
$$ 
over $V_1$, where $G$ is the toric boundary divisor of $P$ which is the sum of the coordinate hyperplanes and the inverse image of $C_1$. In particular, any toric prime divisor $D$ over $V_d$ is also a toric prime divisor over $P$. 

By (2), $\mathbb{T}_{V_i}\simeq \mathbb{T}_{V_{i-1}}\times \mathbb{T}_{\A^1}$ and the morphism 
$\mathbb{T}_{V_i}\to \mathbb{T}_{V_{i-1}}$ is given by projection. Thus, $\mathbb{T}_{V_d}\simeq \mathbb{T}_{V_{1}}\times \mathbb{T}_{\A^{d-1}}$ and the morphism $\mathbb{T}_{V_d}\to \mathbb{T}_{V_{1}}$ is given by projection onto the first factor. 
Identifying $\mathbb{T}_{V_d}$ with the torus $\mathbb{T}_{V_1\times \PP^{d-1}}$, we get the desired birational map $V_d\bir P/V_1$. The assertion about toric prime divisors $D$ follows from the existence of the birational map. 

(5) 
Given a special toric tower as in (2), assume $({Z}_1,E_1)\to (V_1,C_1)$ is a morphism of couples. So the image of $Z_1$ is not contained in $C_1$. Taking $Z_i:=Z_1\times_{V_1}V_i$ and taking 
$E_i\subset Z_i$ to be the inverse image of $E_1$ union the inverse image of $C_i$, 
we can define the \emph{pullback tower} (as in \ref{ss-tower-couples})
$$
({Z}_d,{E}_d)\to \cdots \to ({Z}_1,E_1).
$$ 

Note that $Z_i\to Z_1$ is flat with integral general fibres, by (3) above. Thus $Z_i$ is integral. Moreover, the image of $Z_i\to V_i$ is not contained in $C_i$ because the general fibres of $Z_i\to Z_1$ are not mapped into $C_i$, again by (3).

\subsection{Pullback of special toric towers}
In this subsection we will show that pullback of special toric towers are quite close to being special toric towers.

\begin{prop}\label{p-bchange-special-tower}
Assume that we are given a special toric tower 
\begin{equation}\label{eq-stt-1}
  (V_d,C_d)\to  \cdots \to (V_1,C_1)  
\end{equation}
as in \ref{ss-special-toric-tower}(2), and that $({Z}_1,E_1)\to (V_1,C_1)$ is a morphism of couples from a log smooth couple of dimension one. Let
$$
  ({Z}_d,{E}_d)\to \cdots \to ({Z}_1,E_1)  
$$
be the pullback of (\ref{eq-stt-1}) by base change to $(Z_1,E_1)$ as in \ref{ss-special-toric-tower}(5).
Then for each closed point $z_1\in Z_1$, perhaps after shrinking $Z_1$ around $z_1$, there exists 
a commutative diagram of couples
$$
\xymatrix{
(Z_d,E_d) \ar[d]\ar[r] &(W_d,D_d) \ar[d] \\
(Z_{d-1},E_{d-1}) \ar[d]\ar[r] &(W_{d-1},D_{d-1}) \ar[d] \\
\vdots \ar[d]&   \vdots \ar[d] \\
(Z_{1},E_1)  \ar[r] &   (W_1,D_1)
}
$$
where 
\begin{itemize}
\item the right hand side is a special toric tower, 
\item $\pi_i\colon {Z}_i\to {W}_i$ is an \'etale morphism with $E_i=\pi_i^{-1}D_i$, for each $i$, and 
\item the induced morphism $Z_i\to Z_1\times_{W_1}W_i$ is an open immersion, for each $i$. 
\end{itemize}
\end{prop}
\begin{proof}
\emph{Step 1.}
Shrinking $Z_1$, we can assume it is affine, say $Z_1=\Spec R$. 
First, we construct $\pi_1\colon Z_1\to W_1=\Spec k[t]$. Let $u$ be a local parameter at $z_1$. We can assume it is regular everywhere. Assume $z_1\in E_1$. Then shrinking $Z_1$, we can assume that $E_1=z_1$ and we define $k[t]\to R$ by sending $t$ to $u$ which then gives $\pi_1\colon Z_1\to W_1$. This is \'etale because $u$ is a local parameter. Also $E_1=\pi_1^{-1}D_1$ where $D_1$ is the origin on $W_1$. 

Now assume $z_1\notin E_1$. Shrinking $Z_1$ around $z_1$ we can assume $E_1=0$. Let $Z_1\to W_1$ be the morphism 
given by $k[t]\to R$ sending $t$ to $u-1$. Shrinking $Z_1$, we can assume that the morphism is \'etale near $z_1$ and that we again have $E_1=\pi_1^{-1}D_1$. Note that the induced map $Z_1\to Z_1\times_{W_1}W_1$ is an isomorphism.

\emph{Step 2}. 
Recall the variables $\alpha_2,\dots,\alpha_d$ and the equations $\Phi_2, \dots,\Phi_d$ in the definition of the given tower 
$$
(V_d,C_d)\to  \cdots \to (V_1,C_1).
$$ 
For each $i$, either $\Phi_i=0$ or $\Phi_i=\alpha_i\alpha_i'-\lambda_i$ for some non-zero character $\lambda_i$ in the variables $\alpha_2,\dots, \alpha_{i-1}, t_1,\dots,t_p$. 
In case $\Phi_i=0$, $V_i=V_{i-1}\times \A^1$ and $C_i$ is the inverse image of $C_{i-1}$ plus the vanishing section of the variable $\alpha_i$ on $\A^1$. And in case $\Phi_i=\alpha_i\alpha_i'-\lambda_i$, 
$$
V_i\subset V_{i-1}^\circ\times \A^2
$$ 
is the closed subscheme defined by $\Phi_i$, where $V_{i-1}^\circ\subset V_{i-1}$ is the maximal open subset where $\lambda_i$ is regular, $\alpha_i,\alpha_i'$ are the coordinate variables on $\A^2$, and $C_i$ is the inverse image of $C_{i-1}$.

The given morphism $Z_1\to V_1$ induces a homomorphism 
$$
k[t_1,\dots,t_p]\to R.
$$
Let $s_j$ be the image of $t_j$ under this homomorphism, that is, $s_j$ is the pullback of $t_j$ to $Z_1$ which is non-zero since the generic point of $Z_1$ maps to outside $C_1$ by assumption. 
Then in case $\Phi_i=0$, $Z_i=Z_{i-1}\times \A^1$ and $E_i$ is the inverse image of $E_{i-1}$ plus the vanishing section of $\alpha_i$. And in case $\Phi_i=\alpha_i\alpha_i'-\lambda_i$, 
$$
Z_i\subset Z_{i-1}^\circ\times \A^2
$$ 
is the closed subscheme defined by $\alpha_i\alpha_i'-\lambda_i|_{Z_{i-1}^\circ}$, where $Z_{i-1}^\circ$ is the inverse image of $V_{i-1}^\circ$ to $Z_{i-1}$, and $E_i$ is the inverse image of $E_{i-1}$. Here, $\lambda_i|_{Z_{i-1}^\circ}$ means the pullback of $\lambda_i$ to $Z_{i-1}^\circ$.

\emph{Step 3}. 
We will construct $(W_i,D_i)$ and $\pi_i\colon Z_i\to W_i$, inductively. 
Assume that we have already constructed 
$$
\xymatrix{
(Z_{i-1},E_{i-1}) \ar[d]\ar[r] &(W_{i-1},D_{i-1}) \ar[d] \\
\vdots \ar[d]&   \vdots \ar[d] \\
(Z_{1},E_1)  \ar[r] &   (W_1,D_1)
}
$$
satisfying the properties listed in the proposition.
And assume that the right hand side special toric tower is defined using variables $\beta_2,\dots,\beta_{i-1}$ and equations $\Psi_2,\dots,\Psi_{i-1}$. Assume that for each $j\le i-1$, 
\begin{itemize}
\item $\alpha_j$ is the pullback of $\beta_j$, 
\item if $\Phi_j=0$, then $\Psi_j=0$, and 
\item if $\Phi_j=\alpha_j\alpha_j'-\lambda_j$, then $\Psi_j=\beta_j\beta_j'-\gamma_j$ for some $\gamma_j$.
\end{itemize}
\bigskip

\emph{Step 4}. 
Assume $\Phi_i=0$. Then $Z_i=Z_{i-1}\times \Spec k[\alpha_i]$.  Consider the morphism 
$$
\A^1=\Spec k[\alpha_i]\to \A^1=\Spec k[\beta_i]
$$ 
induced by $k[\beta_i]\to k[\alpha_i]$ which sends $\beta_i$ to $\alpha_i$, where $\beta_i$ is a new variable. Let $W_i=W_{i-1}\times \Spec k[\beta_i]$.
Then the two morphisms $Z_{i-1}\to W_{i-1}$ and $\Spec k[\alpha_i]\to \Spec k[\beta_i]$ induce a morphism $\pi_i\colon Z_i\to W_i$ which is \'etale. Let $D_i$ be the inverse image of $D_{i-1}$ plus the vanishing section of $\beta_i$. Then $\pi_i^{-1}D_i=E_i$. We have then constructed 
$$
\xymatrix{
(Z_i,E_i)  \ar[r]\ar[d] &   (W_i,D_i)\ar[d]\\
(Z_{i-1},E_{i-1}) \ar[r] &(W_{i-1},D_{i-1})
}
$$
which extends the diagram in Step 3. In Step 8, we will show that $Z_i\to Z_1\times_{W_1}W_i$ is an open immersion, so the diagram satisfies all the required properties.

\emph{Step 5}. 
From here to the end of Step 7, assume $\Phi_i=\alpha_i\alpha_i'-\lambda_i$ where $\lambda_i$ is a non-zero character in the variables $\alpha_2,\dots,\alpha_{i-1}$, $t_1,\dots,t_p$, say
$$
\lambda_i=\alpha_2^{m_2}\cdots\alpha_{i-1}^{m_{i-1}}t_1^{n_1}\cdots t_p^{n_p}
$$ 
where $m_j,n_j$ are integers. Then 
$$
\lambda_i|_{Z_{i-1}^\circ}=\alpha_2^{m_2}\cdots\alpha_{i-1}^{m_{i-1}}s_1^{n_1}\cdots s_p^{n_p}
$$ 
where as above $Z_{i-1}^\circ$ is the inverse image of $V_{i-1}^\circ$ under the morphism $Z_{i-1}\to V_{i-1}$, 
and 
$$
Z_i\subset Z_{i-1}^\circ\times \A^2
$$ 
is the closed subscheme defined by $\alpha_i\alpha_i'-\lambda_i|_{Z_{i-1}^\circ}$.

We can write $s_j=e_ju^{c_j}$ where $e_j$ is regular and non-vanishing at $z_1$ 
and $c_j$ is a non-negative integer. Shrinking $Z_1$, we can assume $e_j$ are regular 
everywhere but not vanishing at any point. Thus 
$$
\lambda_i|_{Z_{i-1}^\circ}=\alpha_2^{m_2}\cdots\alpha_{i-1}^{m_{i-1}} u^{\sum c_jn_j} e_1^{n_1}\cdots e_p^{n_p}.
$$ 
Note that if $z_1\not\in E_1$, then $z_1$ is mapped into $V_1\setminus C_1$, so none of the $s_j$ vanishes at $z_1$, hence  $\sum c_jn_j=0$, so $u$ does not appear in $\lambda_i|_{Z_{i-1}^\circ}$.

\emph{Step 6}. 
Consider new variables $\beta_i,\beta_i'$ and $\A^2=\Spec k[\beta_i,\beta_i']$. Let 
$$
\gamma_i=\beta_2^{m_2}\cdots\beta_{i-1}^{m_{i-1}} t^{\sum c_jn_j}
$$  
which is a character on $W_{i-1}$. Recall that $W_1=\Spec k[t]$. Let $W_{i-1}^\circ$ be the maximal open subset where $\gamma_i$ is regular. And let 
$$
W_i\subset W_{i-1}^\circ\times \A^2
$$ 
be the closed subscheme defined by $\Psi_i:=\beta_i\beta_i'-\gamma_i$. Let $D_i$ be the inverse image of $D_{i-1}$ under the projection $W_i\to W_{i-1}$. 

Since $Z_{i-1}\to W_{i-1}$ is \'etale, $\gamma_i|_{Z_{i-1}}$ is regular at a closed point $z$ iff $\gamma_i$ is regular at $w=\pi_{i-1}(z)$, by Lemma \ref{l-etale-morphism-regular-points} (note that $Z_{i-1}$ and $W_{i-1}$ are both normal so we can apply the lemma). 
Now 
$$
\gamma_i|_{Z_{i-1}^\circ}=\alpha_2^{m_2}\cdots\alpha_{i-1}^{m_{i-1}} u^{\sum c_jn_j},
$$
where we use the fact that if $z_1\in E_1$, then $t$ pulls back to $u$, but if $z_1\not\in E_1$, then $\sum c_jn_j=0$.
So $\lambda_i|_{Z_{i-1}^\circ}=\gamma_i|_{Z_{i-1}^\circ}g_i$ where $g_i=e_1^{n_1}\cdots e_p^{n_p}$ is regular and nowhere vanishing. 
Therefore, $\gamma_i|_{Z_{i-1}^\circ}$ is regular, hence $Z_{i-1}^\circ$ is mapped into $W_{i-1}^\circ$ by $\pi_{i-1}$.

\emph{Step 7}. 
Consider the morphism 
$$
\phi_i\colon Z_{i-1}^\circ\times \Spec k[\alpha_i,\alpha_i'] \to W_{i-1}^\circ\times\Spec k[\beta_i,\beta_i']
$$
which sends a closed point $(z,a,b)$ to the point $(\pi_{i-1}(z),a,\frac{b}{g_i(z)})$. 
So, $\beta_i$ pulls back to $\alpha_i$ but $\beta_i'$ pulls back to $\frac{\alpha_i'}{g_i}$. 
By Lemma \ref{l-special-subvar-etale}, we get $\pi_i\colon Z_i\to W_i$ decomposing as  
$$
Z_i\to Z_{i-1}^\circ\times_{W_{i-1}^\circ}W_i\to W_i
$$ 
where the former is an isomorphism and the latter is \'etale. 

Recall that $D_i\subset W_i$ is the inverse image of $D_{i-1}\subset W_{i-1}$. Then since $\pi_{i-1}^{-1}D_{i-1}=E_{i-1}$ and since $E_i$ is the inverse image of $E_{i-1}$, we deduce that $\pi_{i}^{-1}D_{i}=E_{i}$.
We have then constructed 
$$
\xymatrix{
(Z_i,E_i)  \ar[r]\ar[d] &   (W_i,D_i)\ar[d]\\
(Z_{i-1},E_{i-1}) \ar[r] &(W_{i-1},D_{i-1})
}
$$
which extends the diagram in Step 3. Therefore, inductively we can construct the whole commutative diagram in the statement of the proposition.

\emph{Step 8.}
It remains to show that $Z_i\to Z_1\times_{W_1}W_i$ is an open immersion. This holds for $i=1$ because $Z_1\to Z_1\times_{W_1}W_1$ is an isomorphism. Assuming the claim holds for $i-1$, we show that it also holds for $i$. 
We have 
\begin{equation}\label{eq-open-immersion}
   Z_{i-1}\to Z_1\times_{W_1}W_{i-1}\to W_{i-1}
\end{equation}
where the former morphism is an open immersion. If $\Psi_i=0$, then taking the product of (\ref{eq-open-immersion}) with $\A^1$ quickly shows that $Z_i\to Z_1\times_{W_1}W_i$ is an open immersion. So assume that $\Psi_i\neq 0$. 
Then (\ref{eq-open-immersion}) induces 
$$
Z_{i-1}^\circ\to Z_1\times_{W_1}W_{i-1}^\circ\to W_{i-1}^\circ
$$ 
where again the former morphism is an open immersion. Taking base change via $W_i\to W_{i-1}^\circ$ we get 
$$
Z_{i-1}^\circ\times_{W_{i-1}^\circ}W_i\to (Z_1\times_{W_1}W_{i-1}^\circ)\times_{W_{i-1}^\circ} W_i\to W_{i}
$$ 
which can be re-written as 
$$
Z_{i-1}^\circ\times_{W_{i-1}^\circ}W_i\to Z_1\times_{W_1}W_{i}\to W_{i}
$$
and the former morphism is an open immersion. Now the claim follows by recalling from Step 7 that we also have an isomorphism 
$
Z_i\to Z_{i-1}^\circ\times_{W_{i-1}^\circ}W_i.
$
\end{proof}

\subsection{Toroidal divisors on pullback of special toric towers}
We further investigate pullbacks of special toric towers. 
Assume again that we are given a special toric tower 
\begin{equation}\label{eq-stt}
   (V_d,C_d)\to  \cdots \to (V_1,C_1) 
\end{equation}
as in \ref{ss-special-toric-tower}(2), and that $({Z}_1,E_1)\to (V_1,C_1)$ is a morphism of couples from a log smooth couple of dimension one. Let
$$
    ({Z}_d,{E}_d)\to \cdots \to ({Z}_1,E_1)
$$
be the pullback of (\ref{eq-stt}) by base change to $(Z_1,E_1)$ as in \ref{ss-special-toric-tower}(5).

Also recall from \ref{ss-special-toric-tower}(4) that we have a congruent birational map 
$$
(V_d,C_d) \bir (P=\PP^{d-1}_{V_1},G)
$$
over $V_1$, where $G$ is the sum of the coordinate hyperplanes of $P$ and the inverse image of $C_1$.
Pullback by base change to $({Z}_1,E_1)$ induces a 
congruent birational map 
$$
(Z_d,E_d) \bir (P'=\PP^{d-1}_{Z_1},G')
$$
over $Z_1$, where $G'$ is the sum of the coordinate hyperplanes of $P'$ and the inverse image of $E_1$.

\begin{lem}\label{l-lc-place-pullback-special-toric-tower}
Under the above assumptions, $(Z_d,E_d)$ is lc and any lc place of $(Z_d,E_d)$ is an lc place of $(P',G')$.
\end{lem}
\begin{proof}
The statement is local over $Z_1$, so we may fix a point $z_1\in Z_1$ and shrink $Z_1$ around it.
In particular, we can assume that $E_1=0$ if $z_1\notin E_1$ and $E_1=z_1$ if $z_1\in E_1$. 
By Proposition \ref{p-bchange-special-tower}, after further shrinking $Z_1$ around $z_1$, there exist a special toric tower 
$$
(W_d,D_d)\to \cdots \to (W_1,D_1)
$$ 
and a commutative diagram of couples
$$
\xymatrix{
(Z_d,E_d) \ar[d]\ar[r] &(W_d,D_d) \ar[d] \\
(Z_1,E_1)  \ar[r] &   (W_1,D_1)
}
$$
where the horizontal morphisms are \'etale, $E_d$ is the inverse image of $D_d$, and $E_1$ is the inverse image of $D_1$. 
Moreover, the induced morphism $Z_d\to Z_1\times_{W_1}W_d$ is an open immersion.
By \ref{ss-special-toric-tower}(2), $(W_d,D_d)$ is lc as it is a normal toric couple. Thus $(Z_d,E_d)$ is also lc as singularities are determined locally formally. 

Assume $S$ is an lc place of $(Z_d,E_d)$. We will argue that $S$ is an lc place of $(P',G')$. First, $S$ determines an lc place $R$ of $(W_d,D_d)$, making use of Lemma \ref{l-desced-prime-div}. Moreover, by \ref{ss-special-toric-tower}(4), we have a (toric) congruent birational map 
$$
(W_d,D_d)\bir (\overline{P}=\PP^{d-1}_{W_1}, \overline{G})
$$ 
over $W_1$, where $\overline{G}$ is the sum of the coordinate hyperplanes of $\overline{P}$ and the inverse image of $D_1$. Then $R$ is an lc place of $(\overline{P},\overline{G})$ as $R$ is a toric divisor.

Taking pullback of $(W_d,D_d)\bir (\overline{P},\overline{G})$ by base change to $(Z_1,E_1)$ gives a congruent birational map
$$
(W_d'',D_d'')\bir (P''=\PP^{d-1}_{Z_1},G'').
$$
Note that here $W_d''=Z_1\times_{W_1}W_d$.
Since the induced morphism $Z_d\to W_d''$ is an open immersion and $E_d=D_d''|_{Z_d}$, 
we get an open immersion 
$$
Z_d\setminus E_d\to W_d''\setminus D_d''.
$$
This in turn induces a birational map 
$$
Z_d \bir P''
$$ 
and an open immersion 
$$
Z_d\setminus E_d\to P''\setminus G''.
$$ 

Now $S$ is an lc place of $(P'',G'')$ because $S$ maps onto $R$. Moreover, the isomorphism 
$$
Z_d\setminus E_d\to P'\setminus G'
$$
and the open immersion 
$$
Z_d\setminus E_d\to P''\setminus G''
$$
 induce an open immersion 
$$
P'\setminus G'\to P''\setminus G''
$$ 
and a birational map $P'\bir P''$.

Denote the pullback of $K_{P''}+G''$ to $P'$, under the said birational map, by $K_{P'}+\Delta'$.
Then $\Supp \Delta'\le G'$, and since $(P'',G'')$ is lc, we deduce that $\Delta'\le G'$. Therefore, from $K_{P''}+G''\sim 0/Z_1$, we see that 
$$
0\le a(S,P',G')\le a(S,P',\Delta')=a(S,P'',G'')=0,
$$ 
hence $S$ is also an lc place of $(P',G')$.

Note that in fact $(P'',G'')$ is isomorphic to $(P',G')$ in an abstract sense but we do not know if the map $Z_d\bir P''$ is the same as the map $Z_d\bir P'$, so to avoid confusion we have used different notation.
\end{proof}

\subsection{Special toric towers of a tower of families of nodal curves}

\begin{prop}\label{p-local-desc-fix-tower}
Let 
$$
(V_d,C_d)\to \cdots \to (V_1,C_1)
$$ 
be a good tower of families of split nodal curves (as in \ref{ss-tower-family-nodal-curves}) where $(V_1,C_1)$ is log smooth. 
Then there exist finitely many commutative diagrams 
$$
\xymatrix{
(V_d,C_d) \ar[d] & W_d\ar[l]  \ar[dd] \ar[r]  &(V_d'',C_d'') \ar[d] \\
\vdots \ar[d]&  & \vdots \ar[d] \\
(V_{1},C_1)   &  \ar[l] W_1\ar[r] & (V_{1}'',C_1'')
}
$$
 satisfying the following. Let $v_d\in V_d$ be a closed point and $v_i\in V_i$ be its image. 
Then we can choose one of the above diagrams and a closed point $w_d\in W_d$ mapping to
$v_d$ and to $v_i''\in V_i''$ such that 
\begin{itemize}
\item $W_d\to V_d$ and $W_d\to V_d''$ are \'etale  and the inverse images of $C_d$ and 
$C_d''$ coincide near $w_d$, 

\item $W_1\to V_1$ is an open immersion, $W_1\to V_1''$ is \'etale, and the inverse image of $C_1''$ coincides with $C_1|_{W_1}$ near $v_1$, 

\item the tower 
$$
(V_d'',C_d'') \to \cdots \to (V_{1}'',C_{1}'')
$$ 
is a special toric tower as in \ref{ss-special-toric-tower}(2),

\item and
$$
(V_i'',C_i''),v_i'' \to (V_{i-1}'',C_{i-1}''),v_{i-1}''
$$ 
is a local toric model of 
$$
(V_i,C_i),v_i \to (V_{i-1},C_{i-1}),v_{i-1}
$$ 
for each $i>1$ (induced by the morphisms $W_j\to V_j$ and $W_j\to V_j''$).
\end{itemize}
\end{prop}
\begin{proof}
\emph{Step 1.}
We apply induction on $d$. Assume $p=\dim V_1$. Let
$$
V_1''=\A^p=\Spec k[t_1,\dots,t_p]
$$ 
and $C_1''$ be the vanishing set of $t_1\cdots t_p$. 
If $d=1$, then since $(V_1,C_1)$ is log smooth, we can find an open neighbourhood $W_1$ of 
$v_1$ and an \'etale morphism $W_1\to V_1''$ so that $C_1|_{W_1}$ is the inverse image of 
$C_1''$. If $v_1''\in V_1''$ is the image of $v_1$, then 
$(V_1'',C_1''),v_1''$ is a local toric model of $(V_1,C_1),v_1$. Since $V_1$ is quasi-compact, we need only finitely many such $W_1$. 
We then assume $d\ge 2$.

\emph{Step 2.}
We can assume the proposition holds for $d-1$. Then there exist finitely many diagrams 
$$
\xymatrix{
(V_{d-1},C_{d-1}) \ar[d] & W_{d-1}\ar[l]  \ar[dd] \ar[r] &(V_{d-1}'',C_{d-1}'') \ar[d] \\
\vdots \ar[d]&  & \vdots \ar[d] \\
(V_{1},C_1)   &  \ar[l] W_1\ar[r] & (V_{1}'',C_1'')
}
$$
satisfying the properties listed in the proposition. 
Choose one of these diagrams for the point $v_{d-1}\in V_{d-1}$. 
Assume the special toric tower 
$$
(V_{d-1}'',C_{d-1}'') \to \cdots \to (V_{1}'',C_{1}'')
$$ 
is defined by equations 
$\Phi_2,\dots,\Phi_{d-1}$ using variables $\alpha_2,\dots,\alpha_{d-1}$; if $\Phi_i\neq 0$, then we also have another variable $\alpha_i'$. Also, inductively,  $(V_{d-1},C_{d-1})$ is toroidal.

\emph{Step 3.}
By definition of a good tower of families of split nodal curves (see \ref{ss-tower-family-nodal-curves}), $(V_d,C_d)\to (V_{d-1},C_{d-1})$ satisfies the assumptions of Proposition \ref{p-fam-nod-curves-over-toric}. Here we view $(V_{d-1}'',C_{d-1}''),v_{d-1}''$ as a local toric model of $(V_{d-1},C_{d-1}),v_{d-1}$ via the \'etale morphisms $W_{d-1}\to V_{d-1}$ and $W_{d-1}\to V_{d-1}''$. 
By the proof of Proposition \ref{p-fam-nod-curves-over-toric}, there is a local toric model 
$$
(V_d'',C_d''),v_d'' \to (V_{d-1}'',C_{d-1}''),v_{d-1}''
$$ 
of 
$$
(V_d,C_d),v_d \to (V_{d-1},C_{d-1}),v_{d-1}
$$ 
where one of the following two cases occurs. If $V_d\to V_{d-1}$ is smooth at $v_d$, 
then  $V_d''=V_{d-1}''\times \A^1$, $\Phi_d=0$, $C_d''$ is the inverse image of $C_{d-1}''$ plus the vanishing section of a new variable $\alpha_d$ on $\A^1$. 
 But if $V_d\to V_{d-1}$ is not smooth at $v_d$, then   
$$
V_d''\subset V_{d-1}''^\circ\times \A^2
$$ 
 is the closed subscheme defined by some equation
$\Phi_d=\alpha_d\alpha_{d}'-\lambda_d$ where $\lambda_d$ is a non-zero character in  $\alpha_2,\dots, \alpha_{d-1},t_1,\dots,t_p$ and $\alpha_{d},\alpha_{d}'$ are new variables on $\A^2$.
Moreover, in this case, $C_d''$ is the inverse image of $C_{d-1}''$. 
In any case, the tower 
$$
(V_{d}'',C_{d}'') \to \cdots \to (V_{1}'',C_{1}'')
$$ 
is a special toric tower defined by 
$\Phi_2,\dots,\Phi_d$.

\emph{Step 4.}
Now we explain how to get $W_d$. Let 
$$
U_d=V_d\times_{V_{d-1}}W_{d-1} \ \ \mbox{and} \ \ U_d''=V_d''\times_{V_{d-1}''}W_{d-1}.
$$ 
Since $v_d,w_{d-1}$ map to $v_{d-1}$, 
there exists a closed point $u_d\in U_d$ mapping to $v_d,w_{d-1}$. Similarly, there exists a closed point 
$u_d''\in U_d''$ mapping to $v_d'',w_{d-1}$. Moreover, since we viewed $(V_{d-1}'',C_{d-1}''),v_{d-1}''$ as a local toric model of $(V_{d-1},C_{d-1}),v_{d-1}$ via the \'etale morphisms $W_{d-1}\to V_{d-1}$ and $W_{d-1}\to V_{d-1}''$,
we get an induced commutative diagram  
$$
\xymatrix{
\widehat{\mathcal{O}}_{U_d,u_d} \ar[rr] & & \widehat{\mathcal{O}}_{U_d'',u_d''} \\
& \ar[lu] \widehat{\mathcal{O}}_{W_{d-1},w_{d-1}} \ar[ru] & 
}
$$
of completions of local rings 
where the horizontal arrow is an isomorphism. 
Therefore, by  [\ref{Artin}, Corollary 2.6], there is a common \'etale neighbourhood $W_d$ of 
$u_d$ and $u_d''$ giving a commutative diagram 
$$
 \xymatrix{
&W_d \ar[ld]\ar[rd]&\\
U_d\ar[rd] & &U_d'' \ar[ld] \\
&W_{d-1}&
}
$$
where some closed point $w_d\in W_d$ is mapped to $u_d,u_d''$. 
Thus the induced morphisms $W_d\to V_d$ and $W_d\to V_d''$ are \'etale  
fitting into a diagram as in the statement of the proposition.

We show that we can make sure the inverse images of $C_d,C_d''$ coincide near $w_d$. If $V_d\to V_{d-1}$ is not smooth at $v_d$, then near $v_d$, $C_d$ is the inverse image of $C_{d-1}$, and similarly, near $v_d''$, $C_d''$ is the inverse image of $C_{d-1}''$, hence the claim follows as the inverse images of $C_{d-1},C_{d-1}''$ to $W_{d-1}$ coincide near $w_{d-1}$. So assume $V_d\to V_{d-1}$ is smooth at $v_d$. We can assume that $v_d$ belongs to a (unique) horizontal$/V_{d-1}$ component $T_d$ of $C_d$ otherwise the same reasoning applies. According to the proof of Proposition \ref{p-fam-nod-curves-over-toric}, letting $\A^1=\Spec k[\alpha_d]$, there exist an open neighbourhood $\tilde{V}_d$ of $v_d$ and an \'etale morphism $\tilde{V}_d\to V_{d-1}\times \A^1$ so that the inverse image of the vanishing section of $\alpha_d$ to $\tilde{V}_d$ coincides with $T_d$, near $v_d$. Also $V_d''=V_{d-1}''\times \A^1$ and $v_d''$ belongs to the corresponding vanishing section. 
But then base change of $V_{d-1}\times \A^1\to V_{d-1}$ to $W_{d-1}$ is just $U_d''\to W_{d-1}$. Therefore, we get an 
induced \'etale morphism from the neighbourhood $\tilde{U}_d=\tilde{V}_d\times_{V_{d-1}}W_{d-1}$ of $u_d$ onto a neighbourhood of $u_d''$ so that the inverse image of the vanishing section of $\alpha_d$ on $U_d''$ coincides with the inverse image of $T_d$ near $u_d$. We can then replace $W_d$ with $\tilde{U}_d$ which then ensures that the inverse images of $C_d,C_d''$ coincide near $w_d$.    

\emph{Step 5.}
Finally, note that after shrinking $W_d$ around $w_d$ we can assume 
the diagram works for any closed point in the image of $W_d\to V_d$. Indeed, assume $\tilde{w}_d\in W_d$ is a closed point and $\tilde{v}_i\in V_i$ and $\tilde{v}_i''\in V_i''$ are its images. By construction, shrinking $W_d$ around $w_d$ if necessary, we can assume that for each $i\ge 2$, in the commutative diagram 
$$
\xymatrix{
(V_{i},C_{i}) \ar[d] & W_{i}\ar[l]  \ar[d] \ar[r]  &(V_{i}'',C_{i}'') \ar[d] \\
(V_{i-1},C_{i-1})   &  \ar[l] W_{i-1}\ar[r] & (V_{i-1}'',C_{i-1}'')
}
$$
the inverse images of $C_i$ and $C_i''$ coincide near $\tilde{w}_i\in W_i$, and the inverse images of $C_{i-1}$ and $C_{i-1}''$ coincide near $\tilde{w}_{i-1}\in W_{i-1}$, where $\tilde{w}_i,\tilde{w}_{i-1}$ are the images of $\tilde{w}_d$. This shows that 
$$
(V_i'',C_i''),\tilde{v}_i'' \to (V_{i-1}'',C_{i-1}''),\tilde{v}_{i-1}''
$$ 
is a local toric model of 
$$
(V_i,C_i),\tilde{v}_i \to (V_{i-1},C_{i-1}),\tilde{v}_{i-1}.
$$ 

Therefore, we only 
need finitely many diagrams as in the proposition since $V_d$ is quasi-compact and it is covered by the images of the $W_d$. 
\end{proof}


\subsection{Models of bounded toroidalisations}

\begin{proof}[Proof of Theorem \ref{p-local-desc-torif-bnd-fib-II}]
\emph{Step 1.}
The case $d=1$ holds trivially, so assume $d\ge 2$. 
As in the proof of Theorem \ref{t-bnd-torification}, we can reduce to the situation where we  
have a fixed couple $(V,C)$ and a surjective projective morphism $V\to T$ such that we have a morphism $Z\to T$ with $X=Z\times_TV$ and $D=Z\times_TC$.  
We borrow the notation and constructions of that proof; in particular, recall the good tower of families of split nodal 
curves 
\begin{equation}\label{eq-model-tsnc}
    (V_d,C_d) \to \cdots \to (V_1,C_1)
\end{equation}
which is an altering of $(V,C)\to T$ where $(V_1,C_1)$ is log smooth. Also  
recall that $X_1'$ is the normalisation of a component of $Z\times_TV_1$ dominating $Z$, and 
\begin{equation}\label{eq-model-pullback-tsnc}
   (X_d',D_d') \to \cdots \to (X_1',D_1')
\end{equation}
is the pullback of (\ref{eq-model-tsnc}) by base change to $(X_1',\rho^{-1}\{z\})$ where $\rho$ denotes $X_1'\to Z$; see  \ref{ss-tower-family-nodal-curves}(3). By the proof of Theorem \ref{t-bnd-torification}, we have $X_d'=X_1'\times_{V_1}V_d$. Also recall that at the end of that proof we put 
$$
(X',D')=(X_d',D_d') \ \ \mbox{and} \ \ (Z',E')=(X_1',D_1').
$$

Note that the tower (\ref{eq-model-pullback-tsnc}) is the same as the pullback of (\ref{eq-model-tsnc}) by base change to $(X_1',D_1')$ in the sense of \ref{ss-tower-couples}(2): indeed, by definition, $D_i'$ is the support of the sum of the inverse images of $C_i$ and $\rho^{-1}\{z\}$ to $X_i'$; in particular, $D_1'$ is the support of the sum of the inverse images of $C_1$ and $\rho^{-1}\{z\}$; but the inverse image of $C_1$ to $V_i$ is contained in $C_i$, hence $D_i'$ is equal to the support of the sum of the inverse images of $C_i$ and $D_1'$ to $X_i'$ as claimed.

\emph{Step 2.}
Let $x'=x_d'\in X'=X_d'$ be a closed point and let $x_i'\in X_i'$ and $v_i\in V_i$ be its images. We also denote $z'=x_1'$.
 By Proposition \ref{p-local-desc-fix-tower}, there exist finitely many diagrams 
 $$
\xymatrix{
(V_d,C_d) \ar[d] & W_d\ar[l]  \ar[dd] \ar[r] &(V_d'',C_d'') \ar[d] \\
\vdots \ar[d]&  & \vdots \ar[d] \\
(V_{1},C_1)   &  \ar[l] W_1\ar[r] & (V_{1}'',C_1'')
}
$$
satisfying the properties listed in that proposition. By the proposition, we can choose one of these diagrams (for the point $v_d\in V_d$) so that there is a closed point $w_d\in W_d$ mapping to $v_d$ and to $v_i''\in V_i''$ such that 
\begin{itemize}
\item $W_d\to V_d$ and $W_d\to V_d''$ are \'etale  and the inverse images of $C_d$ and 
$C_d''$ coincide near $w_d$, 

\item $W_1\to V_1$ is an open immersion, $W_1\to V_1''$ is \'etale, and the inverse image of $C_1''$ coincides with $C_1|_{W_1}$ near $v_1$,

\item the tower 
$$
(V_d'',C_d'') \to \cdots \to (V_{1}'',C_{1}'')
$$ 
is a special toric tower as in \ref{ss-special-toric-tower}(2),

\item and for each $i>1$,
$$
(V_i'',C_i''),v_i'' \to (V_{i-1}'',C_{i-1}''),v_{i-1}''
$$ 
is a local toric model of 
$$
(V_i,C_i),v_i \to (V_{i-1},C_{i-1}),v_{i-1}.
$$ 

\end{itemize} 
Note that since $w_d$ maps to $v_1$, we see that $v_1\in W_1$. To ease notation, we will replace $V_1$ with $W_1$ 
and further shrink $V_1$ so that $C_1$ is the inverse image of $C_1''$, and shrink $Z'$ near 
$z'$ accordingly.

\emph{Step 3.} 
By \ref{ss-special-toric-tower}(4), we have a diagram of couples 
$$
\xymatrix{
(V_d'',C_d'') \ar[rd] \ar@{-->}[rr] & & (P''=\PP^{d-1}_{V_1''},G'') \ar[ld]\\
& (V_{1}'',C_1'')  & 
}
$$
where $G''$ is the sum of the coordinate hyperplanes and the inverse image of $C_1''$, and the horizontal arrow is 
a congruent birational map. By taking pullback of this diagram by base change to $(V_1,C_1)$, as in \ref{ss-tower-couples}(2), we get 
$$
\xymatrix{
(V_d',C_d') \ar[rd] \ar@{-->}[rr] & & (\tilde P':=\PP^{d-1}_{V_1},\tilde{G}') \ar[ld]\\
& (V_{1},C_1) & 
}
$$ 
where $\tilde G'$ is the support of the sum of the inverse images of $G''$ and $C_1$. Actually, 
$\tilde G'$ is the inverse image of $G''$ which in turn coincides with the sum of the coordinate hyperplanes on $\tilde P'$ and the inverse image of $C_1$: this is because $G''$ is equal to the sum of the coordinate hyperplanes on $P''$ and the inverse image of $C_1''$, and $C_1$ is the inverse image of $C_1''$ to $V_1$. 
Similar reasoning shows that $C_d'$ is the inverse image of $C_d''$. In particular, the horizontal map in the diagram is a congruent birational map.

We get an induced commutative diagram  
$$
\xymatrix{
& W_d\ar[ld]  \ar[rd]  & & \\
(V_d,C_d) \ar[rd] & &(V_d',C_d') \ar[ld] \ar@{-->}[r]& (\tilde P':=\PP^{d-1}_{V_1},\tilde{G}') \ar[lld]\\
& (V_{1},C_1)  & & 
}
$$
where both morphisms from $W_d$ are \'etale at $w_d$ (here we use the fact that $W_d\to V_d''$ and $V_d'\to V_d''$ are \'etale). The inverse images of $C_d,C_d'$ coincide near $w_d$ because the inverse images of $C_d,C_d''$ coincide near $w_d$ and because $C_d'$ is the inverse image of $C_d''$.

\emph{Step 4.}
Taking pullback by base change via $(Z',E')\to (V_1,C_1)$ we get a diagram 
$$
\xymatrix{
& M_\circ'\ar[ld] \ar[rd] & & \\
(X',D') \ar[rd] & & (Y_\circ',L_\circ')\ar[ld] \ar@{-->}[r] & (P'=\PP^{d-1}_{Z'},G')\ar[lld]\\
&(Z',E') &&
}
$$
satisfying the following:
\begin{itemize}
\item $Y_\circ'=Z'\times_{V_1}{V}_d'$,
\item $M_\circ'$ is the irreducible component of 
$
 Z'\times_{V_1}{W}_d
$ 
containing the point $m':=(z',w_d)$, 
\item $m'$ maps to $x'\in X'$, 
\item $M_\circ'\to X'$ and $M_\circ'\to Y_\circ'$ are both \'etale, 
\item the inverse images of $D'$ and $L_\circ'$ to $M_\circ'$ coincide near $m'$, 
\item $G'$ is the sum of the coordinate hyperplanes and the inverse image of $E'$, and 
$$
(Y_\circ',L_\circ')\bir (P'=\PP^{d-1}_{Z'},G')
$$ 
is a congruent birational map, and 
\item $(Y_\circ', L_\circ')$ is lc and any lc place of it is also an lc place of $(P',G')$.
\end{itemize}

We elaborate on some of these properties. 
By construction,
$$
(Y_\circ',L_\circ')\to (Z',E')
$$ 
coincides with the pullback of 
$$
(V_d'',C_d'') \to (V_{1}'',C_{1}'')
$$
via base change by 
$$
(Z',E')\to (V_1'',C_1'').
$$
So by \ref{ss-special-toric-tower}(5), $Y_\circ'$ is indeed equal to $Z'\times_{V_1}{V}_d'$ (rather than just an irreducible component of it). 
On the other hand, note that $z',w_d$ map to the same point $v_1$ of $V_1$, so $m'=(z',w_d)$ indeed belongs to $Z'\times_{V_1}{W}_d$. Also since $X'$ is normal and since 
$$
Z'\times_{V_1}W_d\to X'=Z'\times_{V_1}V_d
$$
is \'etale, $Z'\times_{V_1}W_d$ is normal, so only one of its components $M_\circ'$ contains $(z',w_d)$. Moreover, 
since $X'=Z'\times_{V_1}V_d$, we see that $x'$ can be identified with $(z',v_d)\in Z'\times_{V_1}V_d$.
Thus $m'\in M_\circ'$ maps to $x'$ as $w_d$ maps to $v_d$. 

On the other hand, by construction, $D'$ is the union of the inverse images of $C_d$ and $E'$. And $L_\circ'$ is the union of the inverse images of ${C}_d''$ and $E'$.  Then since the inverse images of $C_d$ and $C_d''$ to $W_d$ coincide near $w_d$ by Step 3, we deduce that the inverse images of $D'$ and $L_\circ'$ to $M_\circ'$ coincide near $m'$.

The claim about $G'$ and the congruent birational map follows from the construction (or, see the discussion prior to Lemma \ref{l-lc-place-pullback-special-toric-tower}).
The last claim follows from Lemma \ref{l-lc-place-pullback-special-toric-tower}.

\emph{Step 5.}
Since $V_d' \to V_1$
is a quasi-projective morphism, it admits a relative projectivisation  
$\overline{V}_d'\to  V_1$. Thus we get 
$$
\xymatrix{
& W_d\ar[ld]  \ar[rd]  & & \\
V_d \ar[rd] & &V_d' \ar[ld] \ar[r]& \overline{V}_d' \ar[lld]\\
& V_{1}  & & 
}
$$
where $V_d'\to \overline{V}_d'$ is the induced open immersion.

We extend the above diagram as follows. 
First let $\overline{W}_d$ be a relative projective compactification of $W_d\to V_d\times_{V_1}\overline{V}_d'$.
This induces projective morphisms  $\overline{W}_d\to V_d$ and $\overline{W}_d\to \overline{V}_d'$.
Now let $\tilde{N}'\to \overline{W}_d$ be a resolution of singularities so that the induced map 
$\tilde N'\bir \tilde P'$ is a morphism. We then get an extended commutative diagram 
$$
\xymatrix{
& W_d\ar[ld]  \ar[rrd] \ar[r] & \overline{W}_d \ar[lld] \ar[rrd] &&& \tilde N' \ar[lll]\ar[d]\\
V_d \ar[ddrr]  &&& V_d' \ar[r]\ar[ddl] & {\overline{V}}_d' \ar[ddll]\ar@{-->}[r]& 
\tilde{P}'=\PP^{d-1}_{V_1}\ar[llldd] \\
 &  & & &&\\
& & V_{1}. &&
}
$$

From this we get a diagram 
$$
\xymatrix{
&M'\ar[ld] \ar[rd] & &N' \ar[ll]\ar[d]\\
(X',D') \ar[rd] & & (Y',L')\ar[ld] \ar@{-->}[r] & P'=\PP^{d-1}_{Z'}\ar[lld]\\
&(Z',E') &&
}
$$
where 
\begin{itemize}
\item $Y'$ is the closure of $Y_\circ'$ in $Z'\times_{V_1}\overline{V}_d'$, and  
$L'$ is the closure of $L_\circ'$ union the inverse image of $E'$, 
\item $M'$ is the closure of $M_\circ'$ in $Z'\times_{V_1}\overline{W}_d$, and
\item $N'$ is the irreducible component of 
$
Z'\times_{V_1}\tilde{N}'
$
mapping onto $M'$.
\end{itemize}

\emph{Step 6.}
Note that we can assume that $Z'\to V_1$ maps the general points of $Z'$ to general points of $V_1$ otherwise $Z'$ maps into a fixed closed subset of $V_1$, so going back to Step 1 we can replace $T,V_1$ and decrease their dimensions. 
Since $\tilde{N}'$ is smooth, $Z'\times_{V_1}\tilde{N}'$ is smooth over the generic point of $Z'$, hence the general fibres of $N'\to Z'$ are smooth (though may not be irreducible).
Thus the irreducible components of the general fibres of $N'\to Z'$ are irreducible components of the general fibres of $\tilde{N}'\to V_1$.

We argue that the diagram obtained in the previous step satisfies the properties listed in the 
statement of Theorem \ref{p-local-desc-torif-bnd-fib-II}. Claims (1),(2) follow from the construction. Claims (3) to (6) follow from Steps 4 and 5. The final claim (7) holds because $X', Y',P',N'$ are obtained from the second diagram of Step 5 
which we can choose among finitely many possibilities and because the divisors $A',H'$ are pullbacks of appropariate divisors $A_{V_d}$ and $H_{\overline{V}_d'}$ on $V_d,\overline{V}_d'$, so
$$
\vol_{/Z'}(A'|_{N'}+H'|_{N'}+G'|_{N'})\le 
\vol_{/V_1}(A_{V_d}|_{\tilde{N}'}+H_{\overline{V}_d'}|_{\tilde{N}'}+\tilde{G}'|_{\tilde{N}'})
$$
taking into account the last paragraph.
\end{proof}


\vspace{2cm}

\small
\textsc{Yau Mathematical Sciences Center, JingZhai, Tsinghua University, Hai Dian District, Beijing, China 100084  } \endgraf
\vspace{0.5cm}
\email{Email: birkar@tsinghua.edu.cn\\}

\pagebreak

\end{document}